\documentclass[12pt, a4paper]{article}

\usepackage{color} 
\usepackage{cite}
\usepackage{enumerate}
\usepackage{graphicx} % to include figures 
\usepackage{caption}
\usepackage{subcaption}
\usepackage{amsmath} % great math stuff
\usepackage{amsfonts} % for blackboard bold, etc
\usepackage{amsthm} % better theorem environments
\usepackage{amssymb} 
\usepackage{amsmath}  
\usepackage{latexsym}
\usepackage{graphicx}

\makeatletter
\renewcommand{\@seccntformat}[1]{%
	\ifcsname format@#1\endcsname
	\csname format@#1\endcsname
	\else
	\csname the#1\endcsname\quad % the default
	\fi
}
\g@addto@macro\appendix{%
	\def\format@section{Appendix \thesection: }%
}
\makeatother

\newtheorem{thm}{Theorem}[section]
\newtheorem{lem}[thm]{Lemma}
\newtheorem{dfn}[thm]{Definition}

\newtheorem{cor}[thm]{Corollary}

\newtheorem*{rem*}{Remark}

\newcommand{\ZZ}{\mathbb{Z}} % for Integers
\newcommand{\FF}{\mathbb{F}}  
\newcommand{\rcr}{Q_n(d,r)}     
\newcommand{\cay}{\mathrm{Cay}}
\newcommand{\bsw}{\mathsf{bw}}
\newcommand{\td}{\mathsf{td}} 
\newcommand{\leap}{L}  
\newcommand{\diam}{\mathsf{diam}} 
\newcommand{\Aut}{\mathrm{Aut}} 
\newcommand{\dist}{\mathsf{dist}}    

\def\e{\boldsymbol{e}}
\def\bo{\boldsymbol{0}_n}
\def\a{\boldsymbol{a}}
\def\b{\boldsymbol{b}}
\def\c{\boldsymbol{c}}
\newcommand{\nd}{\lceil n / d  \rceil}  
   
\newcommand{\mnd}{q} 
\newcommand{\alfndr}{\alpha_{n,d,r}}
\newcommand{\btndr}{\beta_{n,d,r}}
\newcommand{\rhlf}{\lfloor r/2 \rfloor}  
  
\newcommand{\rr}{\lfloor r^2 / 2 \rfloor}

\def\ax{(\a,x)} 
\def\by{(\b,y)} 
\def\cz{(\c,z)} 
\def\oo{(\bo,0)} 
		   
\begin{document}
	
	\title{Recursive cubes of rings as models for interconnection networks}
	\author{Hamid Mokhtar\thanks{Corresponding author. Email: hmokhtar@student.unimelb.edu.au}\; and Sanming Zhou\thanks{Email: sanming@unimelb.edu.au}\\ 
			\small School of Mathematics and Statistics \\ 
			\small The University of Melbourne\\ 
			\small Parkville, VIC 3010, Australia}

	\maketitle
	\openup 0.5\jot 	
\begin{abstract} 
We study recursive cubes of rings as models for interconnection networks. We first redefine each of them as a Cayley graph on the semidirect product of an elementary abelian group by a cyclic group in order to facilitate the study of them by using algebraic tools. We give an algorithm for computing shortest paths and the distance between any two vertices in recursive cubes of rings, and obtain the exact value of their diameters.
We obtain sharp bounds on the Wiener index, vertex-forwarding index, edge-forwarding index and bisection width of recursive cubes of rings. The cube-connected cycles and cube-of-rings are special recursive cubes of rings, and hence all results obtained in the paper apply to these well-known networks.  

\medskip
{{\em Keywords:}  Cayley graph; recursive cube of rings;  cube-connected cycles; interconnection networks; routing; shortest path; diameter; Wiener index; edge-forwarding index; vertex-forwarding index; bisection width}
\end{abstract}

\section{Introduction}

The design and analysis of interconnection networks plays an important role in parallel computing, cloud computing, VLSI, etc. In the literature, many network structures have been proposed and studied \cite{dally2004, bermond2000,zhou2009,thomson2008,cortina1998,shah2001,zhou2006, lai2010} for different purposes. Various factors need to be considered in order to achieve high performance and low construction costs of an interconnection network. Among them, vertex-transitivity, small and fixed node degree, small diameter, recursive construction, existence of efficient routing algorithms are some of the desirable properties \cite{heydemann1989,gan,park2000,laksh1993}. For example, networks with smaller diameters will lead to shorter data transmission delay. The forwarding indices \cite{chung1987,heydemann1997} and bisection width are also well-known measures of performance of interconnection networks \cite{shah2001,yan2009,thomson2008,thomson2010,thomson2014,gan,heydemann1989}. 

It is widely known \cite{heydemann1997} that Cayley graphs are good models for interconnection networks due to their many desirable properties, including vertex-transitivity and efficient routing algorithms. In fact, any Cayley graph admits an all-to-all shortest path routing that loads all vertices uniformly \cite{heydemann1994}, and some Cayley graphs have analogous properties with respect to edges \cite{sole1994, zhou2009}. In the literature, several families of Cayley graphs, including circulants, recursive circulants, hypercubes, cube-connected cycles, cube-of-rings, star graphs, butterflies and orbital regular graphs, have been studied from the viewpoint of routing algorithms \cite{gan, park2000, thomson2008, thomson2010, thomson2014}, diameters, and forwarding indices \cite{gauyacq1997,heydemann1997,sole1994,kosowski2009,heydemann1989,cortina1998,shah2001, thomson2008, thomson2010, thomson2014}. All-to-all routings that uniformly load all edges along with edge-forwarding indices were given in \cite{gauyacq1997} for star graphs and in \cite{Fang2012,sole1994,thomson2008, thomson2010, thomson2014, zhou2009} for a few families of Frobenius graphs. 

Since the class of Cayley graphs is huge, it is not a surprise that not every Cayley graph has all desired network properties. For instance, the degrees of hypercubes and recursive circulants increase with their orders, and the diameters of low degree circulants are larger than the logarithm of their orders. In order to overcome shortcomings of existing graphs, Cayley graphs with better performance are in demand. Inspired by the work in \cite{cortina1998}, an interesting family of graphs, called \emph{recursive cubes of rings}, were proposed as interconnection networks in \cite{sun2000}. A recursive cube of rings is not necessarily a Cayley graph, as shown in \cite{xie2013,hu2005} by counterexamples to \cite[Property 4]{sun2000}. Nevertheless, under a natural condition this graph is indeed a Cayley graph as we will see later. In \cite{choi2008} the vertex-disjoint paths problem for recursive cubes of rings was solved by using Hamiltonian circuit Latin squares, and in \cite{sun2000} the recursive construction of them was given. The diameter problem for recursive cubes of rings has attracted considerable attention: An upper bound was given in \cite[Property 5]{sun2000} but shown to be incorrect in \cite[Example 6]{xie2013}; and another upper bound was given in \cite[Theorem 13]{xie2013} but it was unknown whether it gives the exact value of the diameter. A result in \cite{hu2005} on the diameter of a recursive cube of rings was also shown to be incorrect in \cite{xie2013}.

\subsection{Main results}

%The main results of this paper are as follows.

The purpose of this paper is to conduct a comprehensive study of recursive cubes of rings. As mentioned above, a recursive cube of rings as defined in \cite{sun2000, xie2013} is not necessarily a Cayley graph. We will give a necessary and sufficient condition for this graph to be a Cayley graph (see Theorem \ref{thm:rcriicay}). We will see that, under this condition (given in \eqref{eq:grp}), a recursive cube of rings as in \cite{xie2013} can be equivalently defined as a Cayley graph on the semidirect product of an elementary abelian group by a cyclic group (see Definition \ref{def:rcr}). We believe that this definition is more convenient for studying various network properties of recursive cubes of rings. For example, from our definition it follows immediately that the cube-connected cycles \cite{preparata1981} and cube-of-rings \cite{cortina1998} are special recursive cubes of rings. 

The above-mentioned condition (see \eqref{eq:grp}) will be assumed from Section \ref{sec:shrtpth} onwards. 
In Section \ref{sec:shrtpth}, we give a method for finding a shortest path between any two vertices and a formula for the distance between them in a recursive cube of rings (see Theorems \ref{thm:Pax1} and \ref{thm:Pax2}). In Section \ref{sec:diam}, we give an exact formula for the diameter of any recursive cube of rings (see Theorem \ref{thm:diam}). This result shows that the upper bound for the diameter given in \cite{xie2013} is not tight in general, though it is sharp in a special case. In Section \ref{sec:td}, we give nearly matching lower and upper bounds on the Wiener index of a recursive cube of rings, expressed in terms of the total distance from a fixed vertex to all other vertices (see Theorems \ref{thm:td1} and \ref{thm:td2}). These results will be used in Section \ref{sec:indices} to obtain the vertex-forwarding index (see Theorem \ref{thm:xi}) and nearly matching lower and upper bounds for the edge-forwarding index (Theorem \ref{thm:pi}) of a recursive cube of rings. Another tool for obtaining the latter is the theory \cite{shah2001} of integral uniform flows in orbital-proportional graphs. In Section \ref{sec:bw}, we give nearly matching lower and upper bounds for the bisection width of a recursive cube of rings, which improve the existing upper bounds in \cite{hu2005,sun2000,xie2013}. 

Since the cube-connected cycles \cite{preparata1981} and cube-of-rings \cite{cortina1998} are special recursive cubes of rings, all results obtained in this paper are valid for these well known networks. In particular, we recover a couple of existing results for them in a few case, and obtain new results for them in the rest cases. All results in the paper are also valid for the network $RCR$-$II(d ,r ,n-d )$ \cite{xie2013} with $dr \equiv 0\mod{n}$ (see the discussion in Section \ref{sec:larger}).

Our study in this paper shows that recursive cubes of rings enjoy fixed degree, logarithmic diameter and relatively small forwarding indices in some cases, and flexible choice of order and other invariants when their defining parameters vary. Therefore, they are promising topologies for interconnection networks.

\subsection{Terminology and notation}
\label{sec:term}

All graphs considered in the paper are undirected graphs without loops and multi-edges unless stated otherwise. Since any interconnection network can be modelled as a graph, we use the terms `graph' and `network' interchangeably.  

A \emph{path} of \emph{length} $n$ between two vertices $u$ and $v$ in a graph $X$ is a sequence $u = u_0, e_1, u_1, e_2, \ldots, u_{n-1}, e_n, u_{n} = v$, where $u_0, u_1,  \ldots, u_{n-1}, u_{n}$ are pairwise distinct vertices of $ X $ and $e_1, e_2, \ldots, e_n$ are pairwise distinct edges of $ X $ such that $e_i$ is the edge joining $u_{i-1}$ and $u_{i}$, $1 \le i \le n$. We may simply represent such a path by $u_0, u_1,  \ldots, u_{n-1}, u_{n}$ or $e_1, e_2, \ldots, e_n$. A path between $u$ and $v$ with minimum length is called a \emph{shortest path} between $u$ and $v$.  
The \emph{distance} between $u$ and $v$ in $ X $, denoted by $ \dist(u,v) $, is the length of a shortest path between them in $ X $, and is $\infty$ if there is no path in $ X $ between $u$ and $v$.  The diameter of $ X $ is defined as $\diam(X) := \max_{u,v\in V(X)} \dist(u,v)$. The \emph{Wiener index} of $X$ is defined as $W(X) := \sum_{u,v\in V(X)} \dist(u,v)$, with the sum over all unordered pairs of vertices $u, v$ of $X$. The Wiener index is important for chemical graph theory \cite{bonchev2002} but is difficult to compute in general. It is also used to estimate (or compute) the edge-forwarding index of a network (see \cite{zhou2009, thomson2008, thomson2014}).  

A permutation of $ V(X) $ is called an \emph{automorphism} of $ X $ if it preserves the adjacency and non-adjacency relations of $ X $. The set of all automorphisms of $ X $ under the usual composition of permutations is a group, $\Aut(X)$, called the \emph{automorphism group} of $ X $. If $ \Aut(X) $ is transitive on $V(X)$, namely any vertex can be mapped to any other vertex by an automorphism of $X$, then $ X $ is called \emph{vertex-transitive}. The definition of an \emph{edge-transitive} graph is understood similarly. 

If $X$ is vertex-transitive, then define the \emph{total distance} $\td(X)$ of $X$ to be the sum of the distances from any fixed vertex to all other vertices in $X$. It can be easily seen that, for a vertex-transitive graph $X$, the average distance of $X$ is equal to $\td(X)/(|V(X)|-1)$ and the Wiener index of $ X $ is given by $ W(X) = |V(X)|\td(X)/2 $.  

Let $G$ be a group and $S$ be a subset of $G$ such that $ 1_G\notin S$ and $s^{-1} \in S$ for $s \in S$, where $ 1_G $ is the identity element of $ G $. Then the {\em Cayley graph} on $G$ with respect to the \emph{connection set} $S$, denoted by $\cay(G,S)$, is defined to have vertex set $G$ such that $u, v \in G$ are adjacent if and only if $u^{-1}v \in S$. It is known that $\cay(G,S)$ is connected if and only if $S$ is a generating set of $G$. It is also well known that $G$ acts on itself by left-regular multiplication as a group of automorphisms of $\cay(G,S)$. That is, every $g \in G$ gives rise to an automorphism $\hat{g}: G \rightarrow G,\ u \mapsto g^{-1}u$ of $\cay(G,S)$, and the group of these permutations $\hat{g}$ form a vertex-transitive subgroup of $\Aut(\cay(G,S))$ that is isomorphic to $G$. In particular, $\cay(G,S)$ is vertex-transitive. 

Let $K$ and $H$ be two groups such that $H$ acts on $K$ as a group. This is to say that, for any $k\in K$, $h\in H$, there corresponds an element of $K$ denoted by $\varphi_h(k)$ such that $\varphi_{1_H}(k)=k$, $\varphi_{h_2}(\varphi_{h_1}(k)) = \varphi_{ h_2 h_1}(k)$ and $\varphi_h(k_1k_2) =\varphi_h(k_1)\varphi_h( k_2)$ for any $k, k_1, k_2 \in K$ and $h, h_1, h_2 \in H$. (In other words, $\varphi: h \mapsto \varphi_h$ defines a homomorphism from $H$ to $\Aut(K)$.) The \emph{semidirect product} of $K$ by $H$ with respect to this action, denoted by $K \rtimes_{\varphi} H$, is the group defined on $K \times H$ ($= \{(k, h): k \in K, h \in H\}$) with operation given by 
\begin{equation}
\label{eq:sd} 
(k_1,h_1) (k_2,h_2) = (k_1 \varphi_{h_1}(k_2),h_1h_2). 
\end{equation}
(A few equivalent definitions of the semidirect product exist in the literature. We use the one in \cite[pp.~22--23]{alperin1995} for convenience of our presentation.)

Throughout the paper we assume that $n$, $d$ and $r$ are positive integers with $n \ge 2$ and $n \ge d$, and $\log a$ is meant $\log_2 a$. From Section \ref{sec:shrtpth} onwards we assume that $ r\ge 3 $. 
 
\section{Recursive cubes of rings}
\label{sec:prl}

In this section we give our definition of a recursive cube of rings. This network is essentially the network RCR-II defined in \cite{xie2013}, which in turn is a modified version of the original recursive cube of rings introduced in \cite{sun2000}. However, unlike \cite{sun2000} and \cite{xie2013}, we impose a condition (see (\ref{eq:grp}) below) to ensure that the network is a Cayley graph and so has the desired symmetry. Without this condition a recursive cube of rings does not behave nicely -- it may not even be regular -- as shown in \cite{hu2005, xie2013}. The treatment in our paper is different from that in \cite{sun2000} and \cite{xie2013}: We define a recursive cube of rings (under condition (\ref{eq:grp})) as a Cayley graph on the semidirect product of an elementary abelian $2$-group by a cyclic group. This definition makes the adjacency relation easier to understand and also facilitates subsequent studies of such networks as we will see later.

\subsection{Recursive cubes of rings}
\label{sec:rcr}

Denote by $\e_i$ the row vector of $\FF_2^n$ (the $n$-dimensional vector space over the 2-element field $\FF_2 = \{0, 1\}$) with the $i$th coordinate 1 and all other coordinates $0$, and denote its transpose by $\e_i^\intercal$, $1 \le i \le n$. An important convention for our discussion is that the subscripts of these vectors are taken modulo $n$, so that $\e_0$ is $\e_n$, $\e_{n+1}$ is $\e_{1}$, $\e_{-2}$ is $\e_{n-2}$, and so on.  Define 
$$
M = [\e_2^\intercal,\dots ,\e_{n}^\intercal,\e_1^\intercal]
$$ 
and treat it as an element of the multiplicative group $GL(n, 2)$ of invertible $n \times n$ matrices over $\FF_2$. Then $M^n = I_n$ is the identity element of $GL(n, 2)$ and 
\[
%\label{eq:eM}
\e_{i}M^{j} = \e_{i + j}
\]
for any integers $i$ and $j$. It can be verified that, under the condition  
\begin{equation}
\label{eq:grp}
	d r  \equiv 0\mod{n},
\end{equation}
the mapping
$$
\varphi: \ZZ_r \rightarrow \Aut(\ZZ_2^n) = GL(n, 2), x \mapsto \varphi_x
$$ 
defined by 
\begin{equation}
\label{eq:action}
\varphi_x (\a) = \a M^{dx},\; \a = (a_1, a_2, \ldots, a_{n}) \in \ZZ_2^n, \, x \in \ZZ_r
\end{equation}
is a homomorphism from $\ZZ_r$ to $\Aut(\ZZ_2^n)$. In other words, the rule (\ref{eq:action}) defines an action as a group of the cyclic group $\ZZ_r$ on the elementary abelian 2-group $\ZZ_2^n$. (Since $M^n = I_n$, the exponent $dx$ of $M$ can be thought as taken modulo $n$.) In fact, for any integers $x, y$ with $x \equiv y\mod r$, by (\ref{eq:grp}) and the fact $M^n = I_n$ we have $M^{dx} = M^{dy}$ and so $\a^x$ defined in (\ref{eq:action}) does not rely on the choice of the representative $x \in \ZZ_r$. Moreover, for $\a, \b \in \ZZ_2^n$ and $x, y \in \ZZ_r$, we have $\varphi_0(\a) = \a$, $\varphi_y(\varphi_x(\a)) = \varphi_y(\a M ^{dx}) = (\a M ^{dx}) M^{dy} = \a M^{d(x+y) } = \varphi_{x+y}(\a)$ and $\varphi_x(\a + \b) = (\a+\b) M^{dx} = \a M^{dx} + \b M^{dx} = \varphi_x(\a) + \varphi_x(\b)$. Since the operations of $\ZZ_2^n$ and $\ZZ_r$ are additions, it follows that indeed (\ref{eq:action}) defines an action of $\ZZ_r$ on $\ZZ_2^n$ as a group. 

Define
\begin{equation}
%\label{eq:G}
\nonumber
G := \ZZ_2^n \rtimes_{\varphi} \ZZ_r
\end{equation}
to be the semidirect product of $\ZZ_2^n$ by $\ZZ_r$ with respect to the action (\ref{eq:action}). In view of (\ref{eq:sd}), the operation of $G$ is given by
\begin{equation}
%\label{eq:op}
\nonumber
\ax  \by = (\a + \b M^{dx} , x + y),
\end{equation}
where the second coordinate $x+y$ is taken modulo $r$. It can be verified that the identity element of $G$ is $\oo$ and the inverse of $\ax$ in $G$ is $(-\a M^{-dx} , r - x)$, where $\mathbf{0}_n = (0, 0, \ldots, 0)$ is the identity element of $\ZZ_2^n$.   

\begin{dfn}
	\label{def:rcr} 
	{\em Let $r \ge 1$, $n \ge 2$ and $d \ge 1$ be integers such that $ n\ge d $ and $d r  \equiv 0\mod{n}$. Define $\rcr$ to be the Cayley graph $\cay(G,S)$ on $G = \ZZ_2^n \rtimes_{\varphi} \ZZ_r$ with respect to the connection set
		\begin{equation}
			\label{eq:setS}
			S:= \{(\boldsymbol{0}_{n},1), (\boldsymbol{0}_{n}, r - 1), (\boldsymbol{e}_{1}, 0), (\boldsymbol{e}_{2}, 0), \ldots, (\boldsymbol{e}_{d}, 0)\}.
		\end{equation}	 
		In other words, $\rcr$ has vertex set $G$ such that for any $\ax \in G$ the neighbours of $\ax$ are: 
		\begin{itemize}
			\item[(a)] $\ax (\e_{i},0) = (\a+\e_{i + dx}, x)$, $ 1 \le i \le d$;  
			\item[(b)] $\ax (\boldsymbol{0}_{n},1) = (\a, x + 1)$ and $\ax (\boldsymbol{0}_{n}, r - 1) = (\a, x - 1)$. 
		\end{itemize}
		We call $ \rcr $ a \emph{recursive cube of rings}. 
		
		The edge joining $\ax$ and $(\a+\e_{i+dx}, x)$ is called a {\em cube edge} of $\rcr$ with {\em direction} $\e_{i}$, and $(\a+\e_{i+dx}, x)$ is called a {\em cube neighbour} of $\ax$. 
		
		The edges joining $\ax$ and $(\a, x + 1), (\a, x - 1)$ are two {\em ring edges} of $\rcr$, and these two vertices are the {\em ring neighbours} of $\ax$. 
		
		The cycle $(\a, 0), (\a, 1), \ldots, (\a, r-1), (\a, 0)$ of $\rcr$ with length $r$ is called the $\a$-\emph{ring} of $ \rcr $.}
\end{dfn}

Since $ dr$ is a multiple of $n $ by our assumption, whenever $x \equiv y\mod r$ we have $(\a+\e_{i + dx}, x) = (\a+\e_{i + dy}, y)$, and so $\rcr$ is well-defined as an undirected graph. 
We may think of $\rcr$ as obtained from the $n$-dimensional cube $Q_n$ (with vertex set $\ZZ_2^n$) by replacing each vertex $\a$ by the corresponding $\a$-ring and then adding cube edges by using rule (a) in Definition \ref{def:rcr}. See Figure \ref{fig:xmpl} for an illustration.  

\begin{figure}[hb]
	\centering
	\includegraphics[scale=.4]{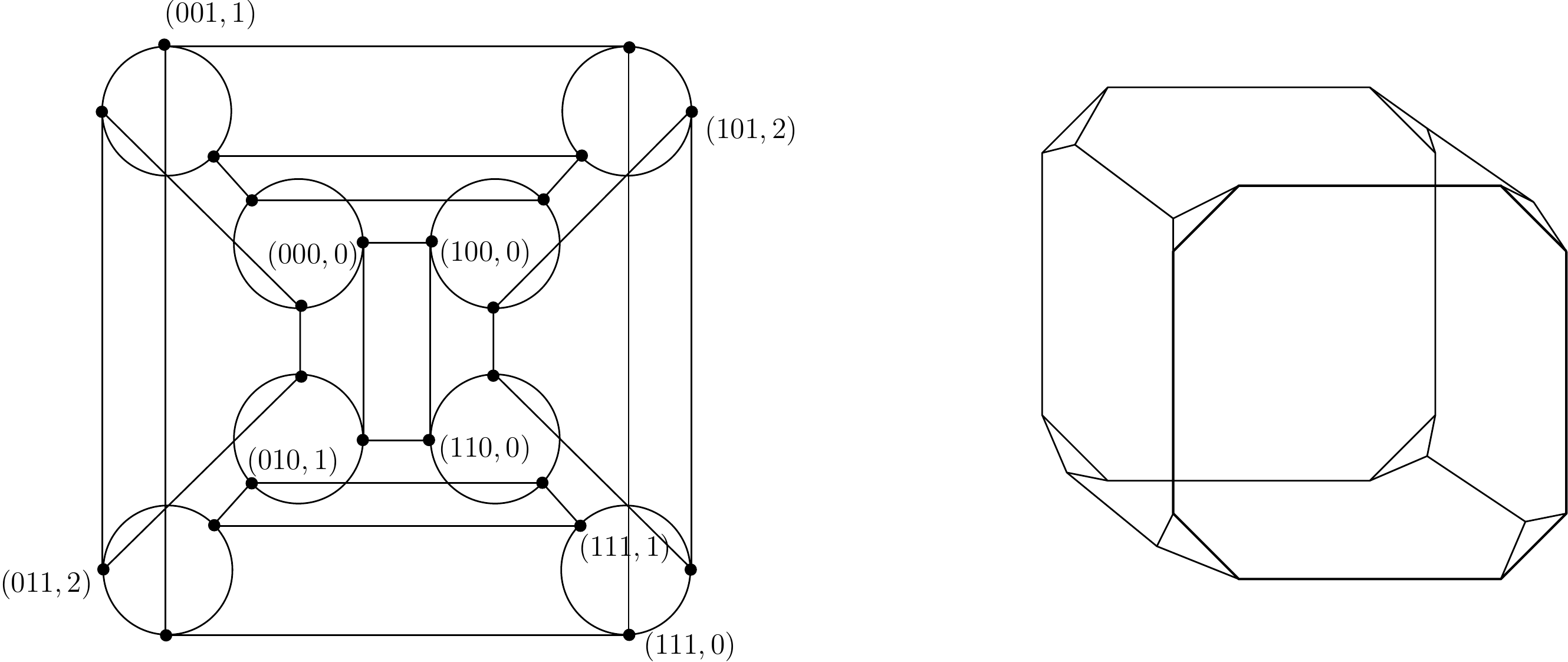}
	\caption{$Q_3(2,3)$ and $Q_3(1,3)$} 
	\label{fig:xmpl}
\end{figure}

The next lemma shows that recursive cubes of rings are common generalizations of three well-known families of interconnection networks, namely, hypercubes, cube-connected cycles $CC_n$ and cube-of-rings $COR(d,r)$ \cite{cortina1998}. $CC_n$ can be defined as the Cayley graph on $\ZZ_2^n \times \ZZ_n $ such that $ \ax$ is adjacent to $\by$ if and only if either $\a = \b $ and $x \equiv y \pm 1 \mod n$, or $\b = \a + \e_{1+x}$ and $x \equiv y \mod n$ (see e.g. \cite{shah2001}). $COR(d,r)$ can be defined \cite[Lemma 2]{cortina1998} as the Cayley graph on the semidirect product of $\ZZ_2^{dr}$ by $\ZZ_r$ with operation given by $\ax  \by = (\a M^{dy}+ \b  , x + y)$, with respect to the connection set $\{(\boldsymbol{0}_{dr},1), (\boldsymbol{0}_{dr}, r - 1), (\boldsymbol{e}_{1}, 0), (\boldsymbol{e}_{2}, 0), \ldots, (\boldsymbol{e}_{d}, 0) \}$.  

\begin{lem}	
	\label{lem:ccc}
	The following hold:
	 \[ 
	 		Q_n \cong Q_{n}(n, 1),\; CC_n \cong Q_{n}(1,n),\;  COR(d, r) \cong Q_{dr}(d, r).
	 \]
	  In other words,  hypercubes, cube-connected cycles and cubes-of-rings are special recursive cubes of rings.
\end{lem}

\begin{proof} 
	When $r=1$, we have $\ZZ_2^n \rtimes_{\varphi} \ZZ_1 \cong \ZZ_2^n$ and $Q_n \cong Q_{n}(n, 1)$ by the definitions of the two graphs.
	By the discussions above, $CC_n$ is the Cayley graph on $\ZZ_2^n \rtimes_{\varphi} \ZZ_n $ with respect to the connection set $\{(\boldsymbol{0}_n , 1) , (\boldsymbol{0}_n , n-1) , (\e_{1}, 0) \}$; hence, $CC_n$ is isomorphic to $Q_{n}(1,n)$. Similarly, $COR(d, r)$ is isomorphic to $Q_{dr}(d, r)$. In fact, the permutation of $\ZZ_2^n\times \ZZ_r$ defined by $(\a, x) \mapsto (\a, x)^{-1} = (-\a M^{-dx} , r-x)$ is an isomorphism from  $COR(d, r)$ to $Q_{dr}(d, r)$. 
\end{proof}	
 
Since hypercubes have been well studied, we will not consider them anymore. Also, we will not consider the less interesting case where $ r=2 $, for which the neighbours $(\a, x + 1)$ and $ (\a, x - 1)$ of $ \ax $ are identical and the ring edges $\{\ax,(\a,x+1) \}$ and $\{\ax, (\a,x-1)\}$ are parallel edges. We assume $ r\ge 3 $ in the rest of this paper.  
 
The following observation follows from the definition of $\rcr$ immediately. 

\begin{lem}	
\label{lem:basic}
	Suppose $ r\ge 3 $. Then $\rcr$ is a connected $(d+2)$-regular graph with $2^n r$ vertices and $ 2^{n-1} r(d+2)$ edges. 
\end{lem}

\begin{proof}
	Only the connectedness requires justification. Since $d r  \equiv 0\mod{n}$, we may assume $dr = tn$ for some integer $t$.  Since $i + dx$ runs over all integers from $1$ to $tn$ when $i$ is running from $1$ to $d$ and $x$ from $0$ to $r-1$, the set $S$ given in (\ref{eq:setS}) is a generating set of $\ZZ_2^n \rtimes_{\varphi} \ZZ_r$. Hence $\rcr$ is connected. 
\end{proof}

It is worth mentioning that in general $ \rcr $ may not be edge-transitive as  $CC_n$ is not edge-transitive \cite{heydemann1989}. 

Denote
\begin{equation}
\label{eq:Fx}
	D(x) := \{i + dx\mod{n}:  1 \le i \le d \},\;\, x \in \ZZ_r. 
\end{equation}

\begin{lem}	
	\label{lem:rpit}
	For any fixed $\a\in\ZZ_2^n$ and $j$ with $1 \le j \le n$, there are exactly $d r/n$ distinct cube edges of $\rcr$ with direction $\mathbf{e}_j$ that are incident to some vertices of the $\a$-ring, namely the edges joining $(\a, x_l)$ and $(\a + \e_j, x_l)$, where $x_l = \lfloor (j+ln-1)/d \rfloor$, $0 \le l < dr/n$.   
\end{lem}

\begin{proof}
	The cube neighbours of $\ax$ are precisely those $(\a + \mathbf{e}_j, x)$ such that $j \in D(x)$.  
	By (\ref{eq:grp}) we have $dr = tn$ for some positive integer $t$. Since $1 \le j \le n$ and $\{ i + dx: 1 \le i\le d, 0\le x<r\} = \cup_{l = 0}^{t-1}\ \{ln+1, \ldots, (l+1)n\}$ is the set of integers from $1$ to $tn$, there are exactly $t$ distinct pairs $(i, x)$ such that $j = i + dx\mod{n}$, namely $(i_l, x_l)$ defined by $x_l = \lfloor (j+ln-1)/d \rfloor$ and $i_l = j+ln -dx_l$, $0 \le l < dr/n$. From this and (\ref{eq:Fx}) the result follows. 
\end{proof}

In the special case when $r = n$, by Lemma \ref{lem:rpit}, there are exactly $d$ cube edges in each direction incident to any given $\a$-ring in $Q_n(d, n)$. Thus $Q_n(d, n)$ can be thought as a generalization of cube-connected cycles; we call it the \emph{$d$-ply cube-connected cycles of dimension $n$}.

\subsection{A larger family of networks}
\label{sec:larger}

We now justify that, under condition (\ref{eq:grp}), $\rcr$ is isomorphic to a recursive cube of rings in the sense of \cite{xie2013}, and vice versa. In \cite{sun2000}, a recursive cube of rings was defined to have vertex set $\ZZ_2^n\times \ZZ_r$ such that $ \ax$ is adjacent to $\by$ if and only if either $\a = \b$ and $x \equiv y \pm 1 \mod r$, or $x \equiv y \mod r$ and $\b = \a + \e_{j}$, where $j = n - x(n-d)-i$ if $n\ge i +  x(n-d)$ and $j = i - d x   \mod{n} $ otherwise for some $1 \le i \le d$. It was claimed in \cite{sun2000} that this is a Cayley graph. However, as shown in \cite{hu2005,xie2013}, in general this graph may not even be regular and so not even be vertex-transitive without condition (\ref{eq:grp}). A modified definition of a recursive cube of rings was given in \cite{xie2013}. We now restate this definition using a different language. $Q_n^-(d,r)$ below is precisely the graph $RCR$-$II(d ,r ,n-d )$ in \cite{xie2013}. 

\begin{dfn}
	\label{def:grcr}
	{\em Let $n \ge 2$, $d \ge 1$ and $r \ge 1$ be integers. Define $Q_n^-(d,r)$ to be the graph with  vertex set $\ZZ_2^n\times \ZZ_r$ such that $\ax$ and $\by$ are adjacent if and only if either $\a=\b$ and $x \equiv y \pm 1 \mod r$, or $\b=\a+\e_{i-dx}$ and $x \equiv y  \mod r$ for some $i$ with $1\le i\le d$.}
\end{dfn}

We call this graph a {\em general recursive cube of rings} and the edge between $\ax$ and $(\a + \e_{i-dx}, x)$ a `cube edge' with `direction' $\e_{i-dx}$. It is known that $Q_n^-(d,r)$ is connected if and only if $ d r \ge n$ \cite[Theorem 3]{xie2013}.   Note that (\ref{eq:grp}) is not required in the definition of $Q_n^-(d,r)$. In \cite[Theorem 9]{xie2013} it was shown that, if (\ref{eq:grp}) is satisfied, then $Q_n^-(d,r)$ is vertex-transitive. The next lemma asserts that under this condition $Q_n^-(d,r)$ is isomorphic to $Q_n(d,r)$ and hence is actually a Cayley graph. 

\begin{lem}
	\label{lem:isom}
	If $d r  \equiv 0\mod{n}$, then $Q_n^-(d,r) \cong \rcr$. 
\end{lem}

\begin{proof}
	Since $dr \equiv 0\ \mod{n}$, similar to (\ref{eq:action}) the rule $\theta_{x}(\a) = \a M^{-dx}$, $\a \in \ZZ_2^n$, $x \in\ZZ_r$, defines an action of $\ZZ_r$ on $\ZZ_2^n$. The operation of the corresponding semidirect product of $\ZZ_2^n$ by $\ZZ_r$ is given by $\ax  \by = (\a + \b M^{-dx}, x + y)$. It can be verified that the Cayley graph on this semidirect product with respect to the same connection set $S$ as in (\ref{eq:setS}) is exactly $Q_n^-(d,r)$. Moreover, the permutation of the set $\ZZ_2^n\times \ZZ_r$ defined by $(\a, x) \mapsto (\a, r-x)$ is an isomorphism from $Q_n^-(d,r)$ to $Q_n(d,r)$.  
\end{proof}

The next result shows that, if $r \ge 3$ and $ n \ge 2d $, then condition (\ref{eq:grp}) is necessary and sufficient for $Q_n^-(d,r)$ to be a Cayley graph. Therefore, all results in the rest of this paper are about $RCR$-$II(d ,r ,n-d )$ with $ dr \equiv 0\mod{n}$. 

\begin{thm}	
\label{thm:rcriicay}
Let $r\ge 3$ and $ n \ge 2d$. Then $Q_n^-(d,r)$ is a connected Cayley graph if and only if $d r  \equiv 0\mod{n}$. 
\end{thm}

A proof of this result will be given in Appendix \ref{app:prfrcrii}.

\section{Shortest paths in $\rcr$}  
\label{sec:shrtpth}

Since $\rcr$ is a Cayley graph, for any $\ax, \by \in G$, if
$$
P_{\ax} : \oo,(\a_1,x_1 ),\ldots,(\a_l,x_l) = \ax
$$
is a path from $ \oo $ to $ \ax$, then
$$
\by P_{\ax}  : \by, \by(\a_1,x_1),\ldots, \by(\a_l,x_l)
$$
is a path from $\by$ to $\by \ax$. Moreover, the former is a shortest path if and only if the latter is a shortest path. Therefore, to find a shortest path between any two vertices, it suffices to find a shortest path from $\oo$ to any $\ax \in G$. This is what we are going to do in this section. 
 
Suppose that $ P$ is a path in $\rcr$ from $ \oo$ to $\ax$ with $ s $ cube edges. Removing these $ s $ cube edges from $ P $ results in $s+1$ subpaths, each of which is a path in a ring and is called a \emph{segment}. Such a segment may contain only one vertex, and this happens if and only if this vertex 
is incident to two cube edges or it is $ \oo$ or $\ax$ and incident to a cube edge on $P$. The first segment must be on the $\bo$-ring, say from $(\bo, 0)$ to $(\bo, x_1)$ for some $x_1 \in \ZZ_r$. If the cube edge on $P$ incident to $(\bo, x_1)$ is in direction $\e_{i_1}$, then the second segment must be on the $\e_{i_1}$-ring from $(\e_{i_1}, x_1)$ to, say, $(\e_{i_1}, x_2)$. In general, for $1 \le t \le s+1$, we may assume that the $t$th segment is on the $(\e_{i_1} + \cdots + \e_{i_{t-1}})$-ring connecting $(\e_{i_1} + \cdots + \e_{i_{t-1}}, x_{t-1})$ and $(\e_{i_1} + \cdots + \e_{i_{t-1}}, x_{t})$ for some $i_1, \ldots, i_{t-1} \in \{1, \ldots, n\}$ and $ x_{t-1}, x_{t} \in \ZZ_r$, where $\e_{i_0}$ is interpreted as $\bo$ and $x_0 = 0$. This implies that, for $1 \le t \le s$, the $t$th cube edge on $P$ is in direction $\e_{i_t}$ and it connects $(\e_{i_1} + \cdots + \e_{i_{t-1}}, x_{t})$ and $(\e_{i_1} + \cdots + \e_{i_{t-1}} + \e_{i_{t}}, x_{t})$ (see Figure \ref{fig:seg}). By the definition of $\rcr$, we have $i_t \in D({x_t})$ for $1 \le t \le s$. So every path $P$ in $\rcr$ from $ \oo$ to $\ax$ determines two tuples, namely, 
$ (x_0, x_1, \ldots, x_s, x_{s+1})$ and $ (i_1, \ldots, i_s),$ where $x_0 = 0$, $x_{s+1} = x$ and $\e_{i_1} + \cdots + \e_{i_{s}} = \a$. Conversely, any two tuples 
\begin{equation}
\label{eq:seqs}
	 \hat{x} = (x_0, x_1, \ldots, x_s, x_{s+1}),\; \hat{i} = (i_1, \ldots, i_s), 
\end{equation}
 such that $i_t \in D({x_t})$ for each $t$, $x_0 = 0$, $x_{s+1} = x$ and $\e_{i_1} + \cdots + \e_{i_{s}} = \a$, give rise to $2^{s+1}$ paths in $\rcr$ from $ \oo$ to $\ax$ with $ s $ cube edges and $s+1$ segments, because the $t$th segment can be one of the two paths from $(\e_{i_1} + \cdots + \e_{i_{t-1}}, x_{t-1})$ to $(\e_{i_1} + \cdots + \e_{i_{t-1}}, x_{t})$ on the $(\e_{i_1} + \cdots + \e_{i_{t-1}})$-ring.   
If we choose the shorter of these two paths for every $t$, then we get a path from $ \oo$ to $\ax$ with shortest length among all these $2^{s+1}$ paths, and this shortest length is $s + l(\hat{x})$ (which is independent of $\hat{i}$), where we define
\[ 
	l( \hat{x}) := \sum_{t=1}^{s+1}  \min\{ |x_{t} - x_{t-1}| , r- |x_{t} - x_{t-1}| \}.
\]   

If a path contains two cube edges with the same direction $\e_i $, then it has a subpath of the form $ (\b,y_0),(\b+\e_i,y_0), \ldots,(\b+\e_i,y_1), (\b+\e_i+\e_{j_1},y_1),\dots,(\b + \e_i + \e_{j_1}+\dots+\e_{j_t},y_t), (\b + \e_{j_1}+\dots+\e_{j_t},y_t)$, and by replacing this subpath with $ (\b,y_0), \ldots,(\b ,y_1),(\b+ \e_{j_1},y_1),\dots,(\b + \e_{j_1}+\dots+\e_{j_t},y_t) $ we obtain a shorter path with the same end-vertices. Therefore, any shortest path from $ \oo $ to $ \ax $ contains exactly one cube edge in direction $ \e_i $ if $ a_i =1 $ and no cube edge in direction $ \e_i $ if $ a_i = 0 $. Thus the number of cube edges in any shortest path from $ \oo $ to $ \ax $ is equal to $\|\a\|$, where $$ \| \a \| := \sum_{i=1}^{n} a_i$$ is the Hamming weight of $\a$. 
 
Define an \emph{$\ax$-sequence} to be a tuple $\hat{x} = (x_0, x_1, \ldots, x_s, x_{s+1})$ with $x_t \in \ZZ_r$ for each $t$ such that $x_0 = 0$, $x_{s+1} = x$, $ s=\|\a\|$,  and for every $i$ with $a_{i}=1$ there is a unique $ t $ with $i \in D({x_t})$. Denote 
\begin{equation}
\label{eq:Lax}
	l\ax :=  \min_{\hat{x}} l( \hat{x}) , 
\end{equation}
with the minimum running over all $\ax$-sequences $\hat{x}$. An $\ax$-sequence achieving the minimum in \eqref{eq:Lax} is said to be \emph{optimal}. Denote by $ \dist(\oo,\ax) $ the distance between $ \oo $ and $ \ax $ in $ \rcr $. The discussion above implies the following results. 
\begin{lem}
\label{lem:axSeq}
	\begin{enumerate}[\rm (a)]
		\item Any $\ax$-sequence $\hat{x}= (x_0, x_1, \ldots, x_s, x_{s+1})$ and any $\hat{i} = (i_1, \ldots, i_s)$ such that $  a_{i_t} =1 $ and $i_t \in D({x_t})$ for $1 \le t \le s$  and $ s=\|\a\| $, give rise to $ 2^{s+1} $ paths  in $ \rcr $ from $ \oo $ to $ \ax $.
		\item The minimum length among the paths obtained from  $\hat{x}$ and $\hat{i} $  is equal to $ \| \a \| + l( \hat{x})$.
		\item $\dist(\oo,\ax) = \|\a\| + l\ax $. 
	\end{enumerate}
\end{lem}

\begin{figure}
\centering
\includegraphics[scale=.45]{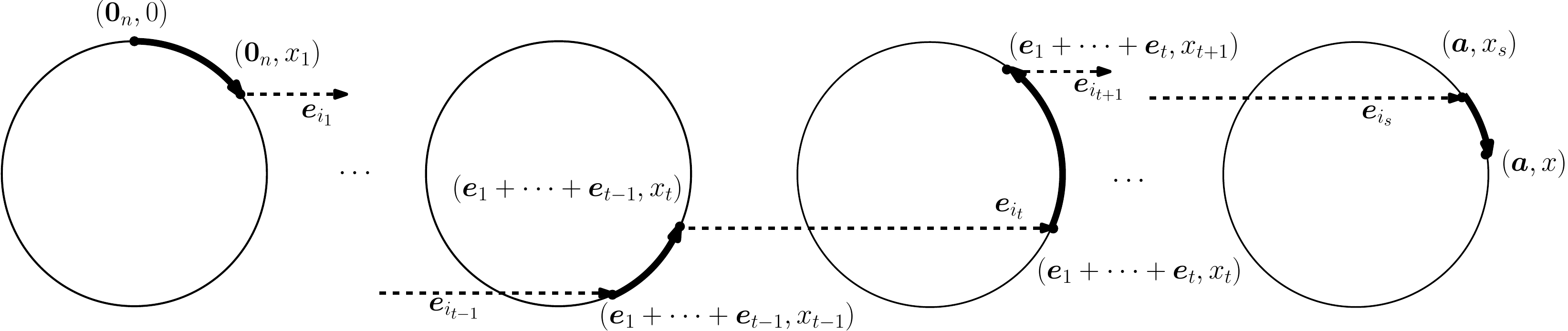}
\caption{Segments of a path from $ \oo $ to $ \ax $}
\label{fig:seg}
\end{figure}

In the rest of this section, we give a method for finding optimal $\ax$-sequences (or equivalently shortest paths from $\oo$ to $\ax$). We need to handle the cases $dr=n$ and $dr \ge 2n$ separately because for each $ i $, by Lemma \ref{lem:rpit}, $ x\in \ZZ_r $ with $ i \in D(x) $ is unique in the former but not in the latter. %In each case, we specify a sequence $ \hat{y}_{\a,x} $ with entries from $\ZZ_r$, and we show that any optimal $\ax$-sequence is obtained by permuting the entries of $ \hat{y}_{\a,x} $ if $ dr =n $ and that of some subsequence of $ \hat{y}_{\a,x} $ if $dr \ge 2n$.  

\subsection{Case $dr = n$} 

In this case, by Lemma \ref{lem:rpit}, for every $1 \le i \le n$ there is a unique $ y \in  \ZZ_r$  such that  $ i \in D(y)$. Hence any $ \ax $-sequence can be obtained from any other $ \ax $-sequence by permuting entries (while fixing the first and last entries). So a sequence $ (y_0, y_1, \dots, y_{s+1}) $ satisfying $y_{t-1} \le y_{t}$, $1 \le t \le s+1$, is obtained by reordering the entries of any $ \ax $-sequence. This sequence is uniquely determined by $ \ax $, with $ y_0=0 $ and $ x=y_{t^*} $ for some $0 \le  t^* \le s+1 $. If $ y_{t^*}<y_{s+1} $, then $ (y_0, y_1, \dots, y_{s+1}) $ is not an $ \ax $-sequence as $ y_{s+1}\ne x $. Denote 
\begin{equation}
\label{eq:yHax}  
	\hat{y}_{\a,x} := (y_0,y_1,\dots, y_{s+1}, y_{s+2}) ,\;y_{s+2} = r. 
\end{equation}
 Define %We now show how to obtain an optimal $ \ax $-sequence from $ \hat{y}_{\a,x}  $.
\[
\leap_1\ax  
:=\left\{
\begin{array}{l l}
\max\{y_{t} - y_{t-1} : 1\le t \le t^* \}, & x \ne 0 , \\ [.2cm]
0 ,   & x=0  \\ 
\end{array} \right .
\] 
\[
\leap_2\ax 
:= \max\{ y_{t } - y_{t-1}: t^*+1 \le t \le s+2 \}.
\] 
Since $ y_{s+2}>y_{s+1} $, $\leap_2\ax \ge 1$, and if $ x\ne 0 $, then $\leap_1\ax \ge 1$. 
Choose $ t $ with $ 1\le t\le t^* $ such that $ y_t-y_{t-1} = \leap_1\ax $ when $x\ne0 $, and $ t=0 $ when $ x=0 $. Now let 
\begin{equation}
\label{eq:xH1}
	\hat{x}^1 = ( y_0,y_1, \dots, y_{t-1}, y_s, \dots, y_{t+1},y_{t},x),
\end{equation}
with assumption that $ y_{t-1}= 0 $ when $ t=0 $.
Choose $ t $ with $ t^*+1\le t \le s+2 $ such that $ y_t-y_{t-1} = \leap_2\ax $, and let
\begin{equation}
\label{eq:xH2}
	\hat{x}^2 =  ( y_0,y_s,\dots, y_{t},y_1,y_2,\dots ,y_{t-1},x ). 
\end{equation} 
(See Figure \ref{fig:drnseq} for an illustration.) It is clear that $ \hat{x}^1 $ and $ \hat{x}^2 $  are $ \ax $-sequences. There exists $ \hat{i}^1 =(i_1,i_2,\ldots,i_{t-1},i_s,\ldots,i_t) $ which together with $ \hat{x}^1 $ satisfies \eqref{eq:seqs}. A path from $ \oo $ and $ \ax $ can be obtained from $ \hat{x}^1 $ and $ \hat{i}^1$ as described above, whose length is $ \|\a\| + l(\hat{x}^1) $ by Lemma \ref{lem:axSeq}. Similarly, a path from $ \oo $ and $ \ax $ can be obtained from  $ \hat{x}^2 $ and $ \hat{i}^2 = (i_s,i_{s-1},\ldots,i_{t},i_1,\ldots,i_{t-1})$, whose length is $ \|\a\| + l(\hat{x}^2) $.
We now show that either $ \hat{x}^1 $ or $ \hat{x}^2 $ is an optimal $ \ax $-sequence and so $ l\ax=\min\{ l(\hat{x}^1),l(\hat{x}^2) \} $.

\begin{thm} \label{thm:Pax1}
 Suppose $dr =n$. Then
\begin{equation}
\label{eq:lax1}
 l\ax = \min \{ r + x - 2 \leap_1 \ax, r +(r- x)  - 2 \leap_2\ax \} 
\end{equation}
and so
\begin{equation}
\label{eq:LpaxCase1}
	\dist(\oo,\ax) = \|\a \|  + \min \{ r + x - 2 \leap_1 \ax, r +(r- x)  - 2 \leap_2\ax \}.
\end{equation}
Moreover, if $l(\hat{x}^1) \le l(\hat{x}^2)$ (respectively, $ l(\hat{x}^2)\le l(\hat{x}^1)  $), then $ \hat{x}^1 $ (respectively, $ \hat{x}^2 $) is an optimal $ \ax $-sequence.
\end{thm}

\begin{figure}
\begin{center}
\includegraphics[scale=.45]{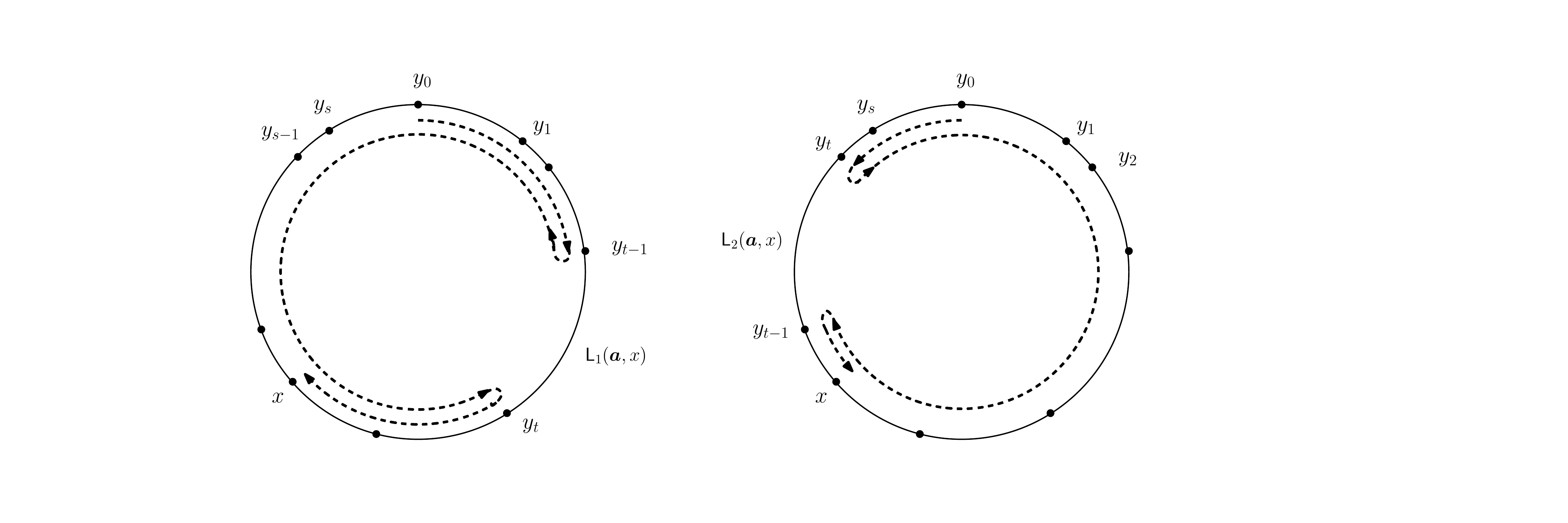} 
\end{center}
\caption{Case $ dr=n $: $ (\a,x) $-sequences $\hat{x}^1$ and  $\hat{x}^2$}
\label{fig:drnseq}
\end{figure}

\begin{proof} 	
	Let $ \hat{x} =(x_0,x_1,\dots,x_{s+1}) $ be an arbitrary $ \ax $-sequence, where $ x_0=0 $, $ x_{s+1} =x$ and $ s=\|\a\| $. From the discussion above, the sequence  $ (y_0,y_1,\dots,y_{s+1}) $ is obtained by reordering the entries of $ \hat{x} $ such that $y_{t-1} \le y_{t}$, for $1 \le t \le s+1$. Let $ C_r $ be the cycle with vertex set $\{ 0,1,2,\ldots,r-1\}$ and edges joining $ 0 $ and $ 1 $, $ 1 $ and $ 2 ,\ldots$, $ r-1 $ and $ 0 $. Any path $ P $ from $ \oo $ to $ \ax $ by using $ \hat{x} $ with minimum length gives rise to a walk $ W $ from $ 0 $ to $ x $ on $ C_r $, obtained by treating each segment of $ P $ as a path on $ C_r $. The length of $ W $ is equal to $ l(\hat{x}) $.
	
	\vspace{.2cm}
	\textsf{Case 1:}~$ W $ contains all edges of $ C_r $. In this case we have $ l(\hat{x}) \ge  \min\{r + x , 2 r - x \} $.
	
	\vspace{.2cm}	
	\textsf{Case 2:}~At least one edge of $ C_r $ is not contained in $ W $. In this case there is exactly one $ t $ with $ 1 \le t \le s+2 $ such that the path $y_{t-1}, y_{t-1}+1,\ldots, y_{t} -1, y_{t} $ is not in $ W $. Conversely, for any $ 1 \le t \le s+2 $, there is a walk $ W $ as above which does not use the path $y_{t-1}, y_{t-1}+1,\ldots, y_{t} -1, y_{t} $. If $ x \ge y_{t }$, then $  l(\hat{x}) \ge 2y_{t-1} + (r- y_{t }) + (x-y_{t }) = r + x- 2 (y_{t}-y_{t-1})$;  if $x \le y_{t-1}$, then $l(\hat{x}) \ge 2(r - y_{t }) + y_{t-1} + (y_{t-1} - x) = 2r - x- 2 (y_{t}-y_{t-1})$.  By the definition of $ \leap_1 \ax $ and $ \leap_2\ax $, the smallest lower bound for $ l(\hat{x}) $ obtained in \textsf{Case 2} is 
\begin{equation}
	\label{eq:LB4lx}
		 l(\hat{x}) \ge \min \{ r+ x-2\leap_1\ax, 2r-x-2\leap_2\ax \} .  
\end{equation}

	Since $ \leap_1\ax \ge 0 $ and $ \leap_2\ax \ge 1 $, we have $ \min \{ r+x , 2r -x \} \ge \min\{ r+x-2\leap_1\ax,2r-x-2\leap_2\ax \} $. In addition, for $ \hat{x}^1 $ and $ \hat{x}^2 $ defined in \eqref{eq:xH1} and \eqref{eq:xH2}, respectively, we have $ l(\hat{x}^1) \le r+ x  - 2 \leap_1\ax $ and $ l(\hat{x}^2) \le 2r-x - 2 \leap_2\ax $. Therefore, by \eqref{eq:LB4lx}, $ l\ax = \min \{ l(\hat{x}^1), l(\hat{x}^2) \} =\min \{ r+ x  - 2 \leap_1\ax, 2r- x - 2\leap_2\ax \} $, which together with Lemma \ref{lem:axSeq} implies  \eqref{eq:LpaxCase1}. Moreover, $ \hat{x}^j $ with $ l(\hat{x}^j) = \min\{ l(\hat{x}^1),l(\hat{x}^2) \} $, $ j\in\{1,2 \}$, is an optimal $ \ax $-sequence.
\end{proof}  

We remark that $ l(\hat{x}^1) \le r+ x  - 2 \leap_1\ax$ and equality holds if $ \hat{x}^1 $ is an optimal $ \ax $-sequence. Similarly, $ l(\hat{x}^2) \le 2r- x -2 \leap_2\ax$ and equality holds if $ \hat{x}^2 $ is an optimal $ \ax $-sequence.

%Lemma \ref{thm:Pax1} gives a shortest $ \ax $-sequence $\hat{x}$ for $P{\ax}$, which is a solution for (\ref{eq:Lax}). An algorithm to obtain an \(\ax\)-optimal in case $ dr =n $ is given in Algorithm \ref{alg:Zk1}. 
%By the above theorem, if $ (\b,y_t), \dots, (\b,y_{t+1}) $ for every $ \b $, is not included in $P{\ax}$, then either $ y_{t+1} - y_t = \leap_1\ax $ and $ y_{t+1} \le x $ or $ y_{t+1} - y_t = \leap_2\ax $ and $  y_{t}\ge x $. 

\begin{rem*}
{\em
In the special case when $ d=1 $ and $ r=n $ (that is, when $ \rcr = CC_n $), Theorem \ref{thm:Pax1} gives rise to \cite[Lemma 1]{shah2001}. 
}
\end{rem*}

\subsection{Case $ d r \ge 2 n $} 
\label{sec:dr2n}
 
Given $ \ax \in G $, let $ \hat{i} = (i_1,i_2,\ldots,i_s) $ be such that $ a_{i_t}=1 $ for  $ 1\le t \le s$ and $ i_{1}< i_2< \cdots <i_{s} $, where $s=\| \a \| $. Since $ dr /n \ge 2 $, by Lemma \ref{lem:rpit} applied to $ j=i_1,i_2,\ldots,i_s $, there exist two $ \ax $-sequences 
\begin{equation}
\label{eq:yzhat} 
\hat{y} = (y_0,y_1,\ldots,y_s,y_{s+1}), \; \hat{z} = (z_{0},z_1,z_{2},\ldots,z_{s},z_{s+1})  
\end{equation}
such that $ y_t = \lfloor (i_t-1)/d \rfloor $, $ z_t = \lfloor (i_t+dr-n-1)/d \rfloor $ and $ i_t \in D({y_t}) \cap D({z_t}) $ for $ 1\le t \le s$. It is clear that $ 0\le y_{t-1} \le y_t \le \lfloor (n-1)/d\rfloor = \nd-1 $ and $r-\nd \le z_{t-1} \le z_t \le r-1  $ for $ 2\le t \le s $.   
Denote $ k = \nd d -n $. Then $ 0\le k\le d-1$. Since $ (i_t -d)/d \le y_t \le (i_t -1 )/d$,  we have $ i_t= k_t + d y_t  $ for some $1\le  k_t \le d$ and every $ 1\le t \le s $. Therefore, $ z_t = \lfloor ( k_t + d y_t +dr-n - 1)/d \rfloor $ and so
\begin{equation}
\label{eq:yz}
z_t = y_t +r -\nd + q_t, \; 1\le t \le s,  
\end{equation}
where $q_t  = \lfloor (k_t+k - 1)/d \rfloor =0 $ or  $1$. Note that if $ n \equiv 0 \mod{d} $, then $ q_t=0 $ for every $ t $. In the following we show how to obtain an optimal $ \ax $-sequence from $ \hat{y} $ and $ \hat{z} $.  

In the case when $ x\le \rhlf $, if $ x<y_s $, then let $ h $, $ 1 \le h \le s$, be such that $ y_{h-1} \le x <y_{h} $.  
Define    
\[ \label{eq:gapax1}
\leap_1 \ax  
:=\left\{
\begin{array}{l l} 
	\max \{ y_{h} -x +q_{h},y_{j} - y_{j-1} +q_{j},\nd\!-y_s :\!h< j \le s    \}, &   x < y_s ,\\ [.2cm]
	\nd - x ,   &   y_{s} \le x \le \rhlf.\\ 
\end{array} \right.  
\]
Similarly, if $ x > \rhlf $ and $ x> z_1 $, then let $ 1 \le l \le s$ be such that $ z_{l} < x \le z_{l+1} $. Define    
\[ \label{eq:gapa	x2}
\leap_2 \ax  
:=\left\{
\begin{array}{l l} 
\!	\max \{ x- z_{l} +\! q_{l},z_{j+1}\! -\!z_{j } +q_{j },  z_1 \!-\! r+\nd\! :\! 0 \le j \le l \}, &  x >\! z_1 ,\\ [.2cm]
	\!	\nd - (r-x),   &   z_{1} \ge x > \rhlf.\\ 
\end{array} \right.  
\] 
If $ x < y_{s} $, then $ \leap_1\ax \ge  y_{h} -x+q_{h} \ge 1$; if $ x>z_1 $, then $ \leap_2\ax \ge  z_{l} -x+q_{l} \ge 1 $. 

\begin{thm}
	\label{thm:Pax2}
	Suppose $dr \ge 2 n$. Then  the following hold. 
%	$$  \dist(\oo,\ax) = \|\a\| +  2\nd - \min\{ x, r-x \} - 2 \leap\ax. $$
		\begin{enumerate}[\rm (a)] 
			\item If $0\le x \le \rhlf$, then  
			$  l \ax  = 2\nd - x - 2 \leap_1\ax  $ and so
			$$ \dist(\oo,\ax) = \|\a\| +  2\nd - x - 2 \leap_1\ax. $$
		 	\item If  $\rhlf < x \le r - 1$, then  
			$  l\ax = 2 \nd-(r -x) - 2 \leap_2\ax  $ and so 
			$$ \dist(\oo,\ax) = \|\a\| +2 \nd-(r -x) - 2 \leap_2\ax .$$
		\end{enumerate}
\end{thm}  

\begin{proof}  
(a) ~Let $ \hat{i} $, $ \hat{y} $ and $ \hat{z} $ be as defined in \eqref{eq:yzhat}. If $ x \ge y_s $, then $ \hat{y} $ is an optimal $ \ax $-sequence since  $ l\ax \ge \min\{ x, r-x \}$ and $ l(\hat{y}) = x $ as $ x\le \rhlf $ and $ 0=y_0\le y_1\le \cdots \le y_s \le y_{s+1}=x $. Hence the length of the path obtained from $ \hat{y} $ and $ \hat{i} $ is $ \|\a\| + x $ and the result follows if $ x\ge y_s $. If $ x<y_s $, then let 
	\begin{equation}
	\label{eq:xHt}	 \hat{x}^t = ( 0, z_s, z_{s-1}, \ldots, z_{t}, y_1,y_2, \ldots,y_{t-1},x ),\;   1 \le t \le s+1 ,
	\end{equation}
	with the understanding that $ \hat{x}^1 = (0,z_s,z_{s-1},\ldots,z_1,x) $ and $ \hat{x}^{s+1} = \hat{y}$. Since $ i_j \in D({y_j}) \cap D({z_j}) $ for $1\le j\le s$, each $ \hat{x}^t$ is an $ \ax $-sequence corresponding to $ \hat{i}^t = (i_s,i_{s-1},\ldots,i_t,i_1,\ldots,i_{t-1}) $,  $ 1 \le t \le s +1$. (See Figure \ref{fig:Zk2} for an illustration.) Thus, for $2\le t \le s$, we have 
	\begin{align*}
	 l(\hat{x}^t) &=  \min\{ z_s,r-z_s \} + \sum_{j=t+1}^{s}\min\{ z_j\!-\!z_{j-1},r\!-\!(z_j\!-\!z_{j-1})\} +  \min\{ |z_t-y_{1}| ,r\!-\!|z_t-y_{1}| \} \\
				& 	 + \sum_{j=2}^{t-1} \min\{ y_{j}-y_{j-1},r-(y_j-y_{j-1}) \} +  \min\{ |x-y_{t-1}|,r-|x-y_{t-1}| \}   .
	\end{align*}
 	 Since $ r >2\nd-2 $ (as $ dr\ge2n $), we have $ z_t-y_{1}\ge z_t-y_t\ge \lfloor r/2\rfloor $ by \eqref{eq:yz}. Moreover, since $ r >2\nd-2 $, $ r-\nd\le z_j \le r-1 $ and $ 0\le y_j \le \nd-1 $, we have $ z_j-z_{j-1} \le \rhlf $ and $ y_{j}-y_{j-1}\le \rhlf $, for $1\le j \le s$. Hence, for $2\le t \le s$, the computation above together with \eqref{eq:yz} gives 
	 \begin{align*}
	 	 l(\hat{x}^t) &= (r-z_t) + \sum_{j=t+1}^{s} ( z_j-z_{j-1}  ) + r-( z_t-y_1) + \sum_{j=2}^{t-1} (y_{j}-y_{j-1}) +|x-y_{t-1}| \\
	 				  &= 2(r - z_t) + y_{t-1} + |x- y_{t-1}| \\  
					  &= 2(\nd - y_{t} - q_t) + y_{t-1} +  |x- y_{t-1}|,\;2\le t \le s.		
	 \end{align*}
	  In addition, $ l(\hat{x}^1) = (r -z_1) +  \min\{  z_1-x,r-( z_1-x) \} $ and $ l(\hat{x}^{s+1}) = y_{s} +  \min\{  |y_s-x|,r-|y_s-x| \} = 2y_s -x $ as $ 0\le x< y_s \le \lfloor r/2\rfloor $. As above, let $ h $ be such that $ 1\le h\le s $ and $ y_{h-1} \le x <y_{h} $. Then 
	  \begin{enumerate}[(i)]
	  	\item $ l(\hat{x}^{s+1}) = 2\nd -x-2(\nd-y_s) $;
	  	\item $ l(\hat{x}^t)=  2 \nd - x  -2( y_{t} -  y_{t-1} + q_t)$, $ h< t \le s $;
	  	\item $l(\hat{x}^h)= 2\nd + x -2(y_{h}+q_h) = 2\nd - x - 2(y_{h}-x+q_h) $;
	  	\item $l(\hat{x}^t)= 2\nd + x -2(y_{t}+q_t) \ge l(\hat{x}^h) $, $ 2\le t < h $; and
	  	\item  $ l(\hat{x}^1) = \min\{ r-x,2(r-z_1)+x \} \ge \min\{ 2\nd -x-2,2(\nd-y_1-q_1)+x \}  \ge \min\{ 2\nd -x-2( y_{h}-x + q_h), 2 \nd-x -2(y_1-x+q_1) \} $  as $r>2\nd-2  $. 	  	
	  	Note that $ y_1-x+q_1\le y_h-x+q_h $ since either $ h=1 $, or $ h>1 $ and $ y_1<y_h $, or $ h>1 $ and $ q_1\le q_h $ if $ y_1 = y_h $ as $ i_1<i_h $. Thus $l(\hat{x}^1)\ge l(\hat{x}^h) $ in each case.
	  \end{enumerate} 
	Therefore,  
	\begin{equation}
	\label{eq:minLthm3}
		\min_{1\le t\le s+1} l(\hat{x}^t)=\min_{h\le t\le s+1} l(\hat{x}^t)= 2 \nd - x  -2\leap_1\ax.
	\end{equation}

	\begin{figure}
		\centering 
			\includegraphics[scale=.45]{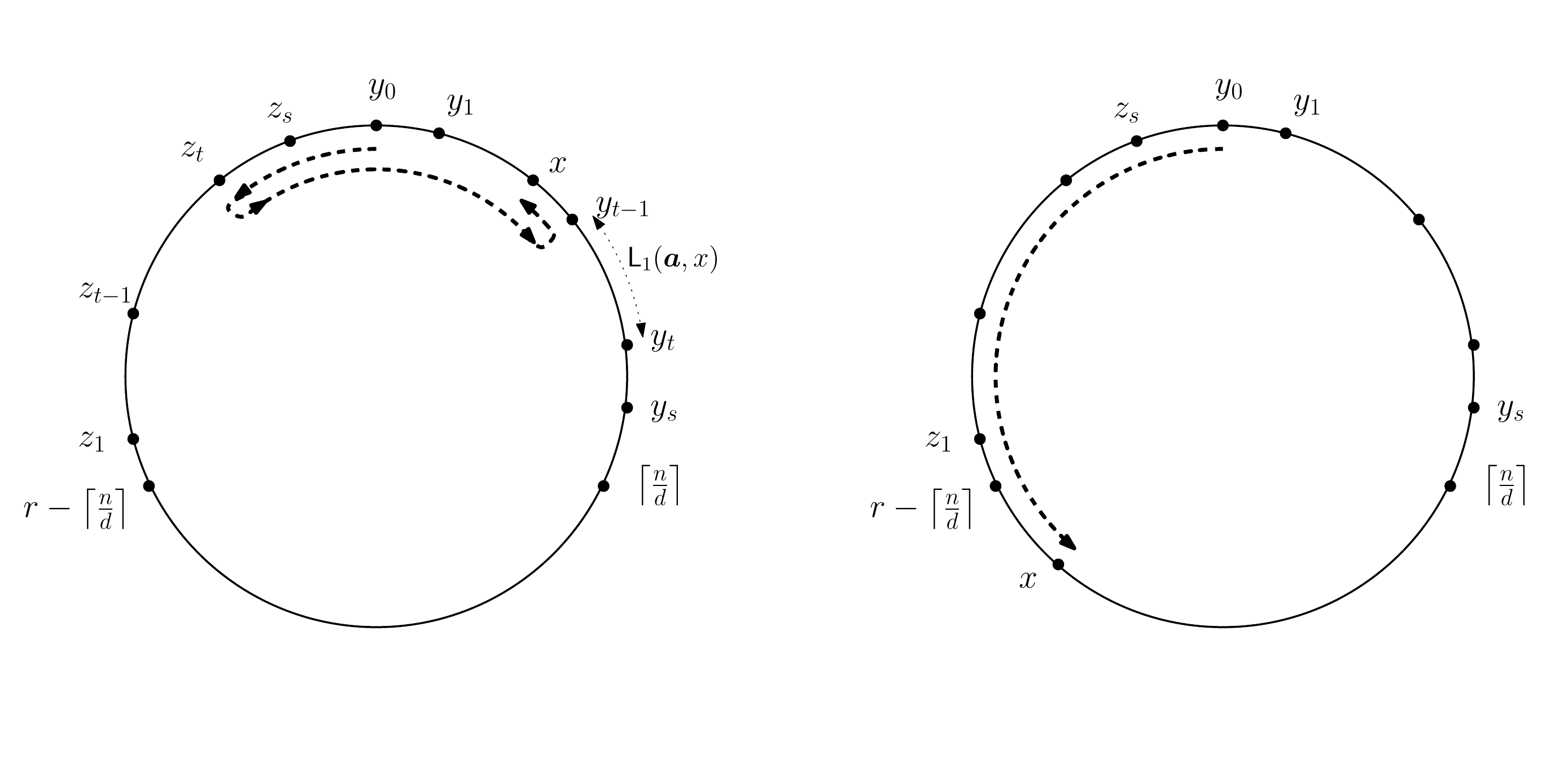}  
		\caption{Case $dr \ge 2n$: $ \ax $-sequences}
		\label{fig:Zk2}
	\end{figure}
	
	Now it remains to show that $l(\hat{w}) \ge 2 \nd - x  -2\leap_1\ax$ for any $ \ax $-sequence $ \hat{w}=(w_0,w_1,\ldots,w_{s+1})\ne \hat{x}^t $, $ 1\le t\le s+1 $. Let $ \hat{k}=(k_1,k_2,\ldots,k_s) $ be a permutation of $ \hat{i} $ such that $ k_j \in D({w_j}) $, $ 1\le j \le s $. By Lemma \ref{lem:rpit}, $ w_j =\lfloor (k_j+l_j n-1)/d \rfloor $ for $ 1\le j \le s $ and some $ 0\le l_j \le dr/n -1 $. 
		
	\vspace{.2cm}
	\textsf{Case 1:}~There exists $ 1\le j \le s $ such that $ w_j $ is contained in neither $ \hat{y} $ nor $\hat{z}$. Then $ 1 \le l_j\le dr/n-2 $ and so $ \lfloor n/d \rfloor \le w_j\le r-\nd $. Since $ x<y_s<\nd $, it follows that $ l(\hat{w}) \ge \min \{ w_j,r-w_j \} +\min \{ |x - w_j| ,r - |x - w_j| \}  \ge 2y_s -x = l(\hat{x}^{s+1})$.

	\vspace{.2cm}
	\textsf{Case 2:}~$ w_j $ is contained in either $ \hat{y} $ or $\hat{z}$ for every $ j $ with $ 1\le j\le s $. If  $ w_j $ is contained in $ \hat{y} $ (respectively, $ \hat{z} $) for all $ j $ with $ 1\le j\le s $, then $ \hat{w} $ is obtained from $ \hat{x}^{s+1} $ (respectively, $ \hat{x}^{1} $) by permuting its entries and so $ l(\hat{w})\ge l(\hat{x}^{s+1}) $ (respectively, $ l(\hat{w})\ge l(\hat{x}^{1}) $). Otherwise, let $ 2\le t \le s $ be the smallest integer such that $ z_t $ is an entry of $ \hat{w} $ and so $ \hat{w} $ must contain $ y_{t-1} $. Then either $ \hat{w}=(0,\ldots,z_t,\ldots,y_{t-1},\ldots,x) $ or $ \hat{w}=(0,\ldots,y_{t-1},\ldots,z_t,\ldots,x) $. Therefore, $ l(\hat{w}) \ge\min\{ (r-z_t) +(r-(z_t-y_{t-1}))+ |x-y_{t-1} |, y_{t-1}+r-(z_t-y_{t-1})+\min\{z_t-x,r+x-z_t \}\}=\min\{ 2(r-z_t) +y_{t-1}+|x-y_{t-1} |, 2y_{t-1}+r-x,2y_{t-1}+2(r-z_t)+x \} $ as $ z_t-y_{t-1} \ge \rhlf $.  Hence $ l(\hat{w})\ge \min\{ l(\hat{x}^t), l(\hat{x}^1),l(\hat{x}^t)\} $. 
	
	In both cases above, there exists some $ 1\le t \le s+1 $ such that $ l(\hat{w}) \ge l(\hat{x}^t) $. 
	So, by \eqref{eq:minLthm3}, $ l\ax =  l(\hat{x}^{t}) $ for some $ h \le t \le s +1$, and any $ \hat{x}^t $ achieving the minimum in \eqref{eq:minLthm3} is an optimal $ \ax $-sequence. 
	Therefore,   the result follows from Lemma \ref{lem:axSeq}.  

	(b) ~The proof is similar to that in case $ (\rm a) $ and so is omitted.   
	\end{proof}	

\begin{rem*}
{\em 
From the proof of Theorem \ref{thm:Pax2}, for any $ \ax\in G$ with $ x \le \rhlf $, $ \hat{x}^t $ given in \eqref{eq:xHt} is an optimal $ \ax $-sequence whenever $ l(\hat{x}^t) = 2\nd-x-2\leap_1\ax$, $ h\le  t\le s+1$. 
%$ t=h $ if $ y_t-x+q_{t} = \leap_1\ax $, or  $ t $ with $ h+1\le t\le s $ and $ y_t-y_{t-1}+q_{t} = \leap_1\ax $, or $ t=s+1 $ if $ y_s \le x $. 
Thus,  $ \hat{x}^t $ and its corresponding $ \hat{i}^t $ give rise to a shortest path from $ \oo $ to $ \ax $ by Lemma \ref{lem:axSeq}. Similarly, for any $ \ax\in G$ with $ x > \rhlf $,  let $ \hat{y} $ and $ \hat{z}$ be as defined in \eqref{eq:yzhat} and $ 1\le l\le s $ be such that $ z_l<x\le z_{l+1 }$. Let
\[ 
\hat{x}^t = (0, y_1, \ldots,y_{t}, z_s, z_{s-1}, \ldots, z_{t+1},x ),\;   1 \le t \le s +1,
\]
where $ \hat{x}^1 $ and $ \hat{x}^{s+1}$ are respectively interpreted as $ \hat{z} $ and $ \hat{y} $. Then $ \hat{x}^t $  is an optimal $ \ax $-sequence whenever $ l(\hat{x}^t) = 2\nd-(r-x) -2\leap_2\ax$, $ 0\le t \le l$.  
%$ t=l $ if $ x-z_l+q_{l} = \leap_2\ax $, or  $ t $ with $ 0\le t\le l $ and $ y_t-y_{t-1}+q_{t} = \leap_2\ax $, or $ t=1 $ if $ x \le z_1 $.  
By Lemma \ref{lem:axSeq}, $ \hat{x}^t $ and $ \hat{i}^t =( i_1,i_2,\ldots,i_{t-1},i_s,\ldots,i_t) $ give rise to a shortest path from $ \oo $ to $ \ax $.
}
\end{rem*}

\section{Diameter of $\rcr$}  
\label{sec:diam}

In \cite[Theorem 5]{sun2000} it was claimed that the diameter of $Q_n^-(d,r)$ (see Section \ref{sec:larger}) is equal to $ n+\lceil (r-3)/2\rceil $, and in \cite[Theorem 3]{hu2005} it was claimed that $\diam(Q_n^-(d,r))\le n+\lfloor r/2\rfloor+1$. As noticed in \cite{xie2013}, these results are incorrect. In \cite[Theorem 13]{xie2013} it was proved that $\diam(Q_n^-(d,r))$ is bounded from above by $ n+\lfloor 3r/2 \rfloor -1 $ if $ r\le 3 $ and $ n+\lfloor 3r/2 \rfloor -2 $ if $ r \ge 4 $. But still the precise value of $\diam(Q_n^-(d,r))$ was unknown. We give the exact value of $\diam(\rcr)$ in the following theorem. Our result shows in particular that the bound $\diam(Q_n^-(d,r)) \le n+\lfloor 3r/2 \rfloor -2$ ($r \ge 4$) is tight when $dr = n$ but not in general (by Lemma \ref{lem:isom}, $Q_n^-(d,r) \cong \rcr$ when $dr = n$). 

\begin{thm}
	\label{thm:diam}
	If $dr =n$, then 
	\[
	\diam(\rcr) = 
	\left\{
	\begin{array}{l l} 
	 n + r, &  \text{if } r = 3  ,\\ [.2cm]
	 n + \lfloor 3r/2\rfloor - 2,&     \text{if }  r \ge 4  .\\ 
	\end{array} \right. 
	\] 
		If $d r \ge 2n$, then $$\diam(\rcr) = n +\max \{ \lfloor{r/2}\rfloor, 2\nd -2 \}.$$ 
\end{thm}
\begin{proof}
	Since $\rcr $ is vertex-transitive, $ \diam(\rcr) = \max_{\ax\in G} \dist(\oo,\ax)  $.
	
	Suppose $dr = n$ first. By (\ref{eq:LpaxCase1}), 
	\[ 
		\dist(\oo,(\a,0)) =   \|\a\| + \min \{r,2r - 2\leap_2(\a,0) \} \le n+ r.
	\]
	We claim that this upper bound is achieved by $(\a,0) =(\boldsymbol{1}_n,0)$. In fact,  for $ 1\le i \le n +2$, we have $ y_i-y_{i-1}=0 $ or $ 1 $ in  $\hat{y}_{\boldsymbol{1}_n,0}$ (given in \eqref{eq:yHax}). Hence $ \leap_2(\boldsymbol{1}_n,0) = 1 $. So $ \min \{ r,2r-2\leap_2(\boldsymbol{1}_n,0)\} =r$ and $ \dist(\oo,(\boldsymbol{1}_n,0)) = n + r $.  
	
	By \eqref{eq:LpaxCase1}, for any $ \ax\in G $ with $ x \ne 0$, since $\leap_1\ax,\leap_2\ax \ge 1$,  we have 
	\[
		 \dist(\oo,\ax) \le  \|\a\| +  \min  \{ r+ x - 2, 2r-x - 2 \} \le  n+ 3\lfloor r/2\rfloor-2.
	\]
	This upper bound is achieved by $\ax = (\boldsymbol{1}_n,\lfloor r/2\rfloor)$. In fact, for $ 1\le i \le n +2$ we have $ y_i-y_{i-1}=0 $ or $ 1 $ in  $\hat{y}_{\boldsymbol{1}_n,\lfloor r/2\rfloor} $ and hence $\leap_1(\boldsymbol{1}_n,\lfloor r/2\rfloor)=\leap_2(\boldsymbol{1}_n,\lfloor r/2\rfloor) = 1$. Note that the maximum of $\min\{ r+x - 2, 2r- x - 2\} $  is $ 3\lfloor r/2\rfloor-2$, which is attained when $x = \lfloor r/2\rfloor$. Thus, $\diam(\rcr)= \max\{n + r,n+ \lfloor 3r/2\rfloor - 2\}$ if $dr=n $, as claimed.	
	 
	% % % % % % % %	
	Now suppose $dr \ge 2n$. For $ \ax\in G $ with $0 \le x \le \lfloor r/2\rfloor$,  $ \dist(\oo,\ax) =   \|\a\| + 2\nd - x -2\leap_1\ax $ by Theorem \ref{thm:Pax2}. 
	If $ \leap_1\ax = \nd -x $, then $ 2\nd - x -2\leap_1\ax = x \le \lfloor r/2\rfloor $; otherwise, $ 2\nd - x -2\leap_1\ax \le 2\nd -2  $ as $ \leap_1\ax \ge 1 $. Thus, 
	\[
		 \dist(\oo,\ax) \le  n+ \max  \{ \lfloor r/2\rfloor, 2\nd -2 \} .
	\]
	
	Note that for any $(\boldsymbol{1}_n,x) $, in $ \hat{y} = (0,y_1,\ldots,y_n,x)$ as given in \eqref{eq:yzhat}, we have $ y_0=0 $, $ y_n=\nd-1 $, and either $y_{t} =y_{t-1}$ or  $y_{t} = y_{t-1}+1 $, $ 1\le t \le n $. In particular, for $ \ax =(\boldsymbol{1}_n,\lfloor r/2\rfloor) $, we have $y_s<\nd \le x $ and so $ \leap_1(\boldsymbol{1}_n,\lfloor r/2\rfloor) = \nd- x $. This implies $  \dist(\oo,(\boldsymbol{1}_n,\lfloor r/2\rfloor)) = n+ \lfloor r/2\rfloor $. 
	On the other hand, for $\ax =  (\boldsymbol{1}_n,0) $, we have $0\le y_{t} -y_{t-1}+ q_{t} \le 2 $, where $q_{t} = \lfloor (k+k_{t}-1)/d \rfloor =0$ or $1$ and  $0\le k,k_t-1 \le d-1 $ as defined in the beginning of Section \ref{sec:dr2n}. % If $y_{t} =y_{t-1}$, then $ y_{t} -y_{t-1}+ q_{t} =0$ or $1$; 
	For any $ t $ with $y_{t} =y_{t-1} +1 $, since $ i_{t}=i_{t-1}+1 $, $ i_{t-1}\in D({y_{t-1}} )$ and $ i_{t}\in D({y_t}) $, we have $ dy_t+k_t =i_{t}=i_{t-1}+1= dy_{t-1}+k_{t-1} +1$, that is, $ d+k_{t}=1+k_{t-1}$. Since $ 1\le k_{t-1},k_t\le d $, we have $ k_{t-1}=d $ and $ k_t=1 $. Therefore, $q_{t}=0$, $  y_{t} -y_{t-1}+ q_{t}=1 $ and so $\leap_1 (\boldsymbol{1}_n,0)= \max_{1\le t\le n} (y_{t} -y_{t-1}+ q_{t}) = 1$. Hence  $ \dist(\oo,(\boldsymbol{1}_n,0)) = n+2\nd -2 $. 
	
	%If $y_{t} =y_{t-1} +1 $, then $ dy_t+k_t = dy_{t-1}+k_{t-1} +1$ for some $ 1\le k_{t-1},k_t\le d$, which implies that $k_{t} =1$, $q_{t}=0$, and so $  y_{t} -y_{t-1}+ q_{t}=1 $. Therefore, $\leap_1 (\boldsymbol{1}_n,0)= \max_{1\le t\le n} (y_{t} -y_{t-1}+ q_{t}) = 1$ and $ \dist(\oo,(\boldsymbol{1}_n,0)) = n+2\nd -2 $. 
	
	Similar to the case $ 0\le x \le \rhlf $, for any $ \ax\in G $ with $\lfloor r/2\rfloor<x\le r-1$, $ \dist(\oo,\ax) \le n+\max  \{ \lceil r/2\rceil-1 ,2\nd -3 \}\le  n+ \max  \{ \lfloor r/2\rfloor, 2\nd -2 \} $.  Therefore, $\diam(\rcr) =  n+\max  \{\lfloor r/2\rfloor   , 2\nd -2 \} $ if $ dr \ge 2n $, and $ \diam(\rcr) $ is attained by $ (\oo,(\boldsymbol{1}_n,\lfloor r/2\rfloor )) $ or $ (\oo,(\boldsymbol{1}_n,0))$. 
\end{proof}

It would be ideal if the diameter of a network is of logarithmic order of its number of vertices. In view of Theorem \ref{thm:diam}, $ \rcr $ has this property when $r = O(n)$. 

 Applying Theorem \ref{thm:diam} to the $d$-ply cube-connected cycles $Q_n(d,n)$, the cube-of-rings \(COR(d,r)\)  and the cube-connected cycles \(CC_n\) (see Lemma \ref{lem:ccc}), we obtain the following corollary. In particular, we recover the formulas for $\diam(COR(d,r))$ and $\diam(CC_n)$ as special cases of  Theorem \ref{thm:diam}. It was claimed in \cite[Theorem 4]{cortina1998} that the diameter of $COR(d,r) $ is $  d (r+1) + \rhlf - 2$ when $ r\ge4 $. Unfortunately, its proof contains a computation error and this formula is incorrect except when $ d=r $. 

\begin{cor}
	\label{cor:diam}
	\begin{enumerate}[\rm (a)]
		\item $\diam(Q_n(d,n)) =n+ \max\{ \lfloor n/2\rfloor, 2\lceil n/d \rceil -2 \}$ for $n\ge d\ge 2$;
		\item $\diam(COR(d,r))=(d+1) r$ if $r=3$, and $ \diam(COR(d,r))=(d +1) r + \rhlf - 2$ if $r\ge 4$;
		\item $\diam(CC_n)=2n$ if $n=3$, and $ \diam(CC_n)=\lfloor 5n /2 \rfloor -2$ if $n\ge 4$ (\hspace{1sp}\cite{fris1997}).
	\end{enumerate} 
\end{cor}

\section{Total distance in $\rcr$}	
\label{sec:td}
    
In this section we give bounds for $\td(\rcr) = \sum_{\ax \in G} \dist(\oo,\ax)$. Since 
\begin{equation}
\nonumber
		 \sum_{\ax\in G} \|\a\| = r \sum_{\a \in \ZZ_2^n} \|\a\|= r \sum_{i=0}^{n} \binom{n}{i} i =  2 ^{n-1} n r,
	\end{equation}
by Lemma \ref{lem:axSeq}, 
\begin{equation}
\label{eq:td}
	\td(\rcr) = \sum_{\ax\in G} (\|\a\| + l\ax) = 2 ^{n-1} n r + \sum_{\ax\in G} l\ax. 
\end{equation}
It remains to estimate $ \sum_{\ax\in G} l\ax$, and for this purpose we will use the notions of integer partitions and $k$-compositions of integers. 

\subsection{Case $dr=n$}	 
\label{sec:td1}

Since $dr = n$, by \eqref{eq:lax1}, $\sum_{\ax \in G} l\ax= \sum_{\ax \in G} \min \{ r + x - 2 \leap_1 \ax, 2r- x-2 \leap_2\ax  \}$. 
In order to give a good estimate of this sum, we will give a lower bound for the number of vertices $ \ax $ such that $\leap_1\ax \le g $ (or $\leap_2\ax \le g $) for a certain $ g\ge1 $.  
 
Let $ 2\le z\le r $ be an integer. For any $\c= (c_{d+1}, c_{d+2}, \dots ,c_{dz}) \in \ZZ_2^{d(z-1)}$, let $$ \hat{w}_{\c} = (w_0,w_1,\ldots,w_s,w_{s+1}) $$ be such that $w_0 =0 $, $w_{s+1}=z $, $ 1\le w_i\le w_{i+1} \le z-1$ for $ 1\le i \le s-1 $, and $ c_i \in D({w_i} )$ for $ 1\le i \le s$, where  $s =\|\c\|$. Since $ dr =n $, $ \hat{w}_{\c} $ is well defined and unique. Define 
\begin{equation}
%\label{eq:Vx}
\nonumber
V(z) := \left\{ \c \in \ZZ_2^{d(z-1)} : \max_{0\le i \le s} (w_{i+1}  - w_i) \le  \lceil \log^2 z \rceil  \right\}.
\end{equation}

\begin{lem}
	\label{lem:Vx}
	 For $2 \le z \le r-1$, we have %with $ z\ge 16 $ ,
	  $$| V(z) | \ge  2^{d(z-1)}\left(1-2z^{1-d \log z}\right).$$
\end{lem}

A proof of this lemma will be given in Appendix \ref{app:lemvx}.
 
\begin{thm}
	\label{thm:td1}
	Suppose $dr=n$. If $ r\ge 2^{9} $, then 
	\begin{equation}
	\label{eq:td1}
	 2^{n-2}r^2(2d+5)\left (1 - \frac{20 \log^2 r}{2n+5r} \right ) \le \td(\rcr) \le  2^{n-2}r^2(2d+5)\left (1- \frac{8(r-1)} {2nr +5r^2 } \right );
	\end{equation}
	and if $ 3\le r<2^{9} $, then
		\begin{equation}
		\label{eq:td1b}
			  2^{n-2} (2nr + r^2) \le \td(\rcr) \le 2^{n-2}  (2nr + 5r^2-8r+8).
		\end{equation}
	
\end{thm}

\begin{proof} 
Note that for any $ \ax $ with $ \a=(a_1,a_2,\ldots,a_n) \in \ZZ_2^n $, the sequence in \eqref{eq:yHax} is exactly the same as that for $ (\a',r-x) $ but with reverse order for all entries except the first and last ones, where $ \a' =(a_n,\ldots,a_2,a_1) $. Therefore, $ \leap_1\ax=\leap_2(\a',r-x) $ and $ \leap_2\ax=\leap_1(\a',r-x) $. Consequently, we have $l \ax= l(\a',r-x) $ and so 
	 \begin{equation}
	 \label{eq:temp1}
	 \sum_{\a \in \ZZ_2^n} l\ax = \sum_{\a \in \ZZ_2^n} l(\a,r-x) .
	 \end{equation}
	
	By \eqref{eq:lax1}, if $x = 0$, then $l\ax\le  r $; % and so $ \sum_{\a \in \ZZ_2^n}l(\a,0)\le 2^nr $; 
	and if $x \ne 0$, then $l\ax \le  r+x-2 $ since $\leap_1\ax \ge 1$. Setting $ \delta=0 $ if $ r $ is odd and $ \delta=1 $ if $ r $ is even, then, by (\ref{eq:temp1}), we have
	\begin{align}
	\nonumber
		\sum_{\ax\in G} l\ax &\le  \sum_{\a \in \ZZ_2^n} \left( 2\sum_{x=1}^{\lfloor (r-1)/2 \rfloor} l\ax + \delta l(\a,\rhlf) + l(\a,0)\right) \\	
		\nonumber
		&\le 2^{n+1}\sum_{x =1}^{\rhlf}(r+x-2)+\delta 2^n(\lfloor 3r/2\rfloor-2) +2^n r \\ 		
		\label{eq:Slax1UB} 
		& \le  2^{n-2}(5r^2-8r +8),
	\end{align}		
	which together with \eqref{eq:td} gives the upper bounds in \eqref{eq:td1} and \eqref{eq:td1b} after straightforward manipulations. 
	 
	It remains to prove the lower bounds in \eqref{eq:td1} and \eqref{eq:td1b}. Observe that \eqref{eq:lax1} and (\ref{eq:temp1}) together yield 
	$$\sum_{\ax \in G} l\ax \ge 2\sum_{x=1}^{\lfloor (r-1)/2 \rfloor} \sum_{\a \in \ZZ_2^n}   \min \{ r+x -2 \leap_1 \ax, 2r- x  - 2 \leap_2\ax  \} .$$ %$$ \ge 2\sum_{x=1}^{\lfloor (r-1)/2 \rfloor} {N_x} (r+x - 2 \leap_1 \ax),$$
	 
	\vspace{.1in}
	\textsf{Case 1:}~Assume first that $r\ge 2^{9}$ and denote $ {h} =\lfloor \log^2 r \rfloor $. Then  $2{h}\le\rhlf $. For any $ \ax $ with $ {h} \le x \le \rhlf-{h}$, since $ \leap_1\ax \ge 1 $ and $ \leap_2\ax \ge 1 $,  if   
	\begin{equation}
	\label{eq:l1l2}
	\leap_1\ax \le \lceil \log^2 x \rceil  \text{ and } \leap_2\ax \le \lceil \log^2 (r-x) \rceil ,
	\end{equation}
	then $ 2\leap_2\ax  - 2\leap_1\ax  \le  2\leap_2\ax-2 \le 2{h} \le r-2 x $, that is, $2r- x - 2 \leap_2\ax  \ge r + x - 2 \leap_1\ax $. 
	
	Denote by ${N_x}$ the number of elements $\ax  \in  G$ such that $ 2r-x-2\leap_2\ax \ge r+x-2  \leap_1\ax$.  
	If $ \b \in V(x) $ and $\b' \in  V(r-x)$, then $ \ax = ((a_1,a_2,\ldots, a_d,\b ,a_{dx+1},\ldots,$  $ a_{dx+d},\b'),x) $  satisfies \eqref{eq:l1l2} for arbitrary $ (a_1,a_2,\ldots, a_d),(a_{dx+1},\ldots, a_{dx+d})\in \ZZ_2^d $. Conversely, for any $ \ax $ satisfying \eqref{eq:l1l2}, we have $ (a_{d+1},a_{d+2},\ldots, a_{dx}) \in  V(x) $ and $(a_{d(x+1)+1},$  $\ldots, a_{dr-1},a_{dr}) \in  V(r-x)$. 	
	Thus, by Lemma \ref{lem:Vx}, for  $ {h} \le x \le \rhlf - {h} $, there are at least $2^{2d} 2^{d(x-1)} (1-2x^{1-d\log x })  2^{d(r-x-1)}(1-2(r-x)^{1-d\log (r-x)})$  elements in $ \ZZ_2^n$  satisfying (\ref{eq:l1l2}), that is, $ {N_x} \ge 2^{n} (1-2x^{1-d\log x })(1-2(r-x)^{1-d\log (r-x)}) $.
	Since  $ {h} \le x \le \rhlf-{h}$, we have  $  x^{1-d\log x } \le  1/(2^{10}r) $ and $  (r-x)^{1-d\log (r-x)} \le 1/(2^{10}r)  $, and therefore ${N_x} \ge 2^{n} (1- 1/(2^{9}r) )^2$. 
% Reasoning: if $ a\le x\le b $, then we have $ x^{\log x}\ge a^{\log a} $. Since  ${h} \le x $, we have $ x^{(d\log x)-1}\ge h^{(d\log h)-1} $ and hence $  x^{1-d\log x}\le  h^{1-d\log h} $. It can be verified that $ \dfrac{2}{(\log^2 r)^{d\log^2r}-2}\le \dfrac{ 1}{2^{10}r} $. Thus, $  x^{1-d\log x } \le  1/(2^{10}r) $. 
	This together with $ r + x - 2\leap_1\ax\ge(r+x)(1-2\lceil\log^2 x\rceil /(r+x)) \ge (r+x)(1 -2\lceil \log^2 h\rceil /(r+h)) $ implies that for $ r\ge 2^{9} $,
	\begin{align}
	\nonumber	
			\sum_{\ax\in G} l\ax & \ge 2\sum_{x={h}}^{\lfloor (r-1)/2 \rfloor -{h}}   {N_x} (r + x - 2\leap_1\ax)  \\ \nonumber		
			& \ge 2\sum_{x={h}}^{\lfloor (r-1)/2 \rfloor -{h}} 2^{n} (1-1/(2^9r))^2 (r+x)(1 -  2\lceil \log^2 h\rceil /(r+h)) \\ \nonumber
			& \ge 2^{n+1} (1- 2.03(\log^2 h)/r)\sum_{x={h}}^{\lfloor (r-1)/2 \rfloor -{h}} ( r + x) \\ \nonumber			
		& = 2^{n-2}(1-2.03(\log^2 h)/r) \left( {5r^2} - (20r-8){h} - 10r\right) \\ \label{eq:Slax1R}
		&\ge  5.2^{n-2}r^2(1-(4\log^2 r/r)). 	 
	\end{align}
 %This together with (\ref{eq:Slax1UB}) implies that, for $ r \ge 2^{9} $,
%	\begin{equation}
%	\label{eq:Slax1R}
%		 5\cdot 2^{n-2}r^2(1 - 4(\log^2 r)/r ) \le \sum_{\ax\in G} l\ax \le 2^{n-2}(5r^2-8r+8) . 
%	\end{equation}
	 Combining this with (\ref{eq:td}), we obtain the lower bound in \eqref{eq:td1}. %	Note that $1/3< 1 - 4(\log^2 r)/r < 1  $ for $ r\ge 2^{9} $. 
	
	\vspace{.1in}
	\textsf{Case 2:}~Now assume $3\le r<2^{9}$. Note that $ l(\a,0)\ge 2 $ except for the $ 2^d $ vertices$ (\a,0)$ with $ \a=(a_1,\ldots,a_d,0\ldots,0) $ for which $ \leap_2(\a,0)= r $. Hence $ \sum_{\a \in \ZZ_2^n}l(\a,0) \ge 2(2^n-2^d)\ge 2^n $ by $ 2d\le n $ and \eqref{eq:lax1}. If $1\le x \le r/2$, then $l\ax \ge x $ since $ \leap_1\ax\le x $ and $ \leap_2\ax\le r-x $.  Therefore, 
	\begin{equation}
	\label{eq:Slax1b} 
		\sum_{\ax \in G} l\ax \ge 2 \sum_{\a \in \ZZ_2^n} \sum_{x=1}^{\lfloor r/2\rfloor } x+\sum_{\a \in \ZZ_2^n}l(\a,0) \ge 2^{n+1} \sum_{x=1}^{\lfloor r/2\rfloor } x +2^n \ge 2^{n-2} r^2.
	\end{equation}
This together with \eqref{eq:td} implies the lower bound in (\ref{eq:td1b}).
\end{proof} 
 
\subsection{Case $dr \ge  2n$} 
\label{sec:tdc2}

Similar to Section \ref{sec:td1}, in order to estimate $ \td(\rcr) $ we will give a lower bound for the number of vertices $ \ax $ with $\leap_1\ax \le g$ for a certain $ g\ge 1 $. For this purpose we will consider integer sequences $ 0=x_0<x_1 < \dots<x_k\le x_{k+1}=m $ such that $  \max_{1\le i \le k+1} ( x_i -x_{i-1}) \le g$. We call a solution $ (r_1,r_2,\dots,r_k) $ to the equation:     
\[
r_1 + r_2 + \dots + r_k = m
\]
with all $r_i$'s positive integers a \emph{$k$-composition} of $m$. It is known that the number of $k$-compositions of $m$ is $\binom{m-1}{k-1}$. Any $k$-composition $ (r_1,r_2,\dots,r_k)  $ of $m$ gives rise to a sequence $ 0=x_0<x_1 < \dots<x_k = x_{k+1}=m $, where $ x_i=\sum_{j=1}^{i}r_j$, $1\le i \le k $.
Clearly, $ \max_{1\le i \le k+1} ( x_i -x_{i-1}) =\max_{1\le i\le k} r_i $. Also any $(k+1)$-composition $ (r_1,r_2,\dots,r_k,r_{k+1})  $ of $m$  gives rise to a sequence $ 0=x_0<x_1 < \dots<x_k < x_{k+1}=m $, where $ x_i=\sum_{j=1}^{i}r_j$, $1\le i \le k $, which satisfies  $ \max_{1\le i \le k+1} ( x_i -x_{i-1}) =\max_{1\le i\le k+1} r_i $. 
Hence, for any fixed $ k $ with $ \lceil m/g \rceil \le k \le m $, the number of sequences $ 0=x_0<x_1 < \dots<x_k = x_{k+1}=m $ (respectively, $ 0=x_0<x_1 < \dots<x_k < x_{k+1}=m $) with $  \max_{1\le i \le k+1} ( x_i -x_{i-1}) \le g $ is equal to the number of $k$-compositions $ (r_1,r_2,\dots,r_k)  $ (respectively, $(k+1)$-compositions $ (r_1,r_2,\dots,r_{k+1})  $) of $m$ with $1\le r_i \le g$ for each $ i $. 

Given integers $ a<b $ and $ k $ with $ (b-a)/ \lceil \log(b-a)\rceil \le k\le b-a $, we defined an \emph{$ [a,b]_k $-sequence} to be an integer sequence $0=x_0<x_1 < \dots<x_k \le x_{k+1}=m $ such that $  \max_{1\le i \le k+1} ( x_i -x_{i-1}) \le \lceil \log(b-a)\rceil $. 
The next lemma is a key step towards an asymptotic formula for $ \td(\rcr) $ to be given in Theorem \ref{thm:td2}.

\begin{lem}
	\label{lem:avr}
	Given an integer $ m\ge 9$, for $k$ with $m/\lceil\log{m}\rceil \le k \le m$ let $ b_k $ be the number of $ [0,m]_k $-sequences. Then, for $g= \lceil \log m \rceil$ and any real number $ z\ge2 $,  
	\[
		  \left( 1-(2/(z+1))^{\lceil\log{m}\rceil} \right) (z+1)^m \le \sum_{k=\lceil  m/g \rceil}^m b_k {z}^k \le ({z}+1)^m .
	\]	
\end{lem}

We postpone the proof of this technical lemma to Appendix \ref{app:lemavr}. 
 
Set 
\begin{equation}
 \label{eq:fndr}
\alfndr:=\left\{
\begin{array}{l l}
  {(12\nd^{3/2}\log({2\nd}))}/{(2nr+r^2+{8\nd^2})} , & \text{if } \nd \ge 100,\\ [.2cm]
   {8\nd^2}/{(2nr+r^2+8\nd^2)}, & \text{if } \nd < 100.
\end{array}\right . 
\end{equation} 

Note that $ 0<\alfndr<1 $ and $ \alfndr $ can be arbitrarily small for sufficiently large $ 2^nr $.  

\begin{thm} 
	\label{thm:td2}
	Suppose $d r \ge 2n$. Then 
	\begin{align}
	\label{eq:td2}
		2^{n-1} \left( n r + \rr +4 \nd^2\right) (1-\alfndr) & \le \td(\rcr)
		\\ \nonumber
		& \le 2^{n-1} \left( n r + \rr +4 \nd^2\right).  
	\end{align}  
\end{thm}
 
\begin{proof} 
Set $ \mnd :=\nd $ in this proof. For any $ \ax\in G $, we have $ \leap_1\ax= \leap_2(\a',r-x)$ in view of (\ref{eq:yzhat}) and so $ l\ax=l(\a',r-x) $, where $ \a=(a_1,a_2,\ldots,a_n) $ and $ \a'=(a_n,\ldots,a_2,a_1) $. So $\sum_{\a \in\ZZ_2^n} l\ax = \sum_{\a \in\ZZ_2^n} l(\a,r-x)  $ for any $1\le x \le \lfloor (r-1)/2\rfloor $.    
Setting $ \delta=0 $ if $ r $ is odd and $ \delta =1 $ if $r$ is even, we then have
\begin{equation}
\label{eq:totalSumtd2}
 	\sum_{\ax\in G} l\ax = 2 \sum_{\a\in \ZZ_2^n}\sum_{x=0}^{\lfloor (r-1)/2\rfloor}l\ax - \sum_{\a\in \ZZ_2^n}l(\a,0) + \delta \sum_{\a\in \ZZ_2^n}l(\a,\lfloor r/2\rfloor).  
\end{equation}
Note that $\sum_{\a\in \ZZ_2^n} l(\a,0) \le \sum_{\a \in \ZZ_2^n}(2\mnd-2)< 2^{n+1} \mnd$ and when $ r $ is even, $\sum_{\a\in \ZZ_2^n}l(\a,\lfloor r/2\rfloor)=\sum_{\a \in \ZZ_2^n} (r/2) = 2^{n-1} r $. 

Denote  
\[
	V := \{\ax\in G: y_s\le x  \le \lfloor (r-1)/2\rfloor,\; \mbox{where $ \hat{y} $ is as in \eqref{eq:yzhat}}\},
\]
where $s=\|\a\|$. So $ \ax\in V $ if and only if either   $x \ge \mnd-1 $, or $0\le x \le \mnd-2 $ and $ a_i=0 $ for $d(x+1)<i\le n $. Thus, for any $\ax \in V$, we have $  \leap_1\ax = \mnd -x $ and so $l\ax = x$ by Theorem \ref{thm:Pax2}.   
	 %If  $\ax \notin V$, then $l\ax = 2\mnd -x-2\leap_1\ax $. 	
	Therefore, 
	\begin{equation}
%	\label{eq:sumF1}
	\label{eq:sumF2}
	\sum_{\a\in \ZZ_2^n}\sum_{x=0}^{ \lfloor (r-1)/2\rfloor}l\ax 
      = A_1 + A_2 + A_3,
	\end{equation}
where 
$$
A_1 = \sum_{x=0}^{ \mnd - 2 }	\sum_{\ax \notin V} ( 2 \mnd -x-2\leap_1\ax ),
$$
\begin{equation}
	\label{eq:tsV1}
A_2 = \sum_{x =0}^{ \mnd -2 } \sum_{\ax \in V} x = \sum_{x =0}^{ \mnd - 2} 2^{d(x+1)}x = \frac{2^{d \mnd }}{(2^d - 1)^2}\left( (2^d-1) (\mnd - 2) -1 \right) +\frac{2^{2d}}{(2^d - 1)^2},
\end{equation}
\begin{equation}
	\label{eq:tsV2}
A_3 = \sum_{\a\in \ZZ_2^n} \sum_{x = \mnd -1 }^{\lfloor (r-1)/2 \rfloor} x = 2^{n-1}  \left\lfloor {(r-1)}/{2} \right\rfloor \left\lfloor {(r+1)}/{2} \right\rfloor - 2^{n-1}   \left(\mnd^2 -3 \mnd +2 \right).
 \end{equation} 
In \eqref{eq:tsV1} we used the fact that, for a fixed $x$ with $0 \le x \le \mnd -2$, the number of elements  $ \ax \in V $ is equal to $2^{d(x+1)}$. Since $ 2^{d \mnd}/(2^d-1) \le 2^n$, we have 
\begin{equation}
	\label{eq:tsV3}
2^{n-d+1}\mnd \le A_2 \le 2^n (\mnd - 2)+4.
\end{equation} 
		%  just consider cases $ d=1 $ and $ d\ge 2 $. %(\ref{eq:tsV1}) $\approx \frac{2^{d \mnd}}{2^d -1} ( \mnd -\alpha)  +\beta^2 $ $ =2^n \beta(\mnd-\alpha)+\beta^2$, for some $2<\alpha<3 $ and $1<\beta \le 4/3$; And ${2^{d \mnd}}/{(2^d -1)} \le 2^{n}$. %(because $ 2^{d(\mnd)-n} \le (2^d -1) $ )  %	since $1<2^{d}/(2^d - 1) \le 4/3$. 
Since $\leap_1\ax \ge 1$, we have
	\begin{equation}
	\label{eq:ts3UBtd2}
		A_1 \le \sum_{\a \in \ZZ_2^n} \sum_{x =0}^{ \mnd - 2 } \left( 2 \mnd - x - 2 \right) 
		\le  2^{n} \sum_{x=0}^{\mnd -2}  \left(2\mnd-x-2 \right) = 2^{n-1} (3 \mnd^2 - 5 \mnd + 2) < 3\cdot 2^{n-1} \mnd^2.
	\end{equation} 

	Combining this with \eqref{eq:totalSumtd2}, \eqref{eq:sumF2}, \eqref{eq:tsV2} and \eqref{eq:tsV3}, we obtain 
	\begin{equation}
	\label{eq:ublax2}
		\sum_{\ax\in G }l\ax \le 2^{n-1}\left(\lfloor {r^2 }/{2} \rfloor + 4 \mnd^2 \right),
	\end{equation}
	which together with \eqref{eq:td} gives the upper bound in \eqref{eq:td2} for all possible $ q $. 

	Now we give a lower bound for $A_1$ and thus a lower bound for $ \td(\rcr)$.  
		
	\vspace{.1in}
	\textsf{Case 1:}~$ \mnd\ge 100 $. 
	For a fixed $x$ with $0\le x \le  \mnd -\sqrt{\mnd} $, denote $g= \lceil \log(\mnd -x-1) \rceil$. (Note that $ \mnd-x-1\ge9 $ and $ g>3 $ as $ \mnd\ge100 $.) Denote by $ W_{x,k} $ the set of $ [x,\mnd-1]_k $-sequences. Denote by $N_{x,k}$ the number of vertices $\ax$ with $ \hat{y}=(y_0,y_1,\ldots,y_t,y_{t+1},\ldots,y_{t+k},x) $ such that $ y_t\le x $ and the sequence $ x<y_{t+1}<\cdots<y_{t+k}\le \mnd-1 $ belongs to $ W_{x,k} $. Note that for any such $ \ax $, we have $ \ax \notin V $ and $ \leap_1\ax \le g+1 $. For a fixed sequence $ y_{t+1},\ldots,y_{t+k} $, the number of vertices $ \ax $ with $ \hat{y} $ as above is $ 2^{d(x+1)} (2^d-1)^k $ since $ (a_{1+dy_{j}},\ldots,a_{d+dy_{j}} )\ne (0,0,\ldots,0) $ for $t+1\le j \le t+k$. 
	 Thus, for any $x$ with $0\le x \le  \mnd -\sqrt{\mnd} $ and $ k $ with $(\mnd-x-1)/g \le k \le \mnd-x-1 $, we have
	\[
	\label{eq:zduk}
		{N_{x,k}} =  2^{d(x+1)} (2^d-1)^k |W_{x,k}|.
	\]

	On the other hand, $2\mnd - x - 2\leap_1\ax \ge 2\mnd-x-2(g+1)=  (2\mnd - x)(1- (g+1)/(\mnd -x/2 )) \ge (2\mnd - x)(1- (\log\mnd)/\sqrt{\mnd} ) $. Denote $ l_1=\lceil (\mnd -x-1)/g \rceil $  and $ l_2=\mnd-x-1 $. Applying Lemma \ref{lem:avr} to $ m=\mnd-x-1 $, we have
	\begin{equation}
	\nonumber
	\begin{array}{lll}
		 A_1 & \ge & \sum_{x=0}^{\mnd-\sqrt{\mnd} }\sum_{k=l_1}^{l_2} {N_{x,k}} (2\mnd-x) \left(1- \frac{\log\mnd}{\sqrt{\mnd}}\right)  \\ [0.2cm]
		&=& \sum_{x=0}^{\mnd -\sqrt{\mnd}} 2^{d(x+1)} (2\mnd-x) \left(1- \frac{\log\mnd}{\sqrt{\mnd}}\right) \sum_{k=l_1}^{l_2} (2^d-1)^k |W_{x,k}| \\ [0.2cm]
	& \ge &	 2^{d\mnd }\left(1- \frac{\log\mnd}{\sqrt{\mnd}}\right) \left(1-\frac{2}{2^{d\log{\mnd}}}\right) \sum_{x =0}^{\mnd -\sqrt{\mnd}}  (2\mnd-x)\\ [0.2cm]
	 & \ge & 3\cdot 2^{n-1}\mnd^2  \left(1- \frac{\log\mnd}{\sqrt{\mnd}}\right) \left(1-\frac{2}{\mnd^{d}}\right)\left(1-\frac{ \sqrt{\mnd}-1}{\mnd}-\frac{\mnd-2{\sqrt{\mnd}}}{2\mnd^2} \right)\\ [0.2cm]
% 	 & \ge & 3\cdot 2^{n-1}\mnd^2  \left(1- \frac{\log\mnd}{\sqrt{\mnd}}\right) \left(1-\frac{2}{\mnd^{d}}\right)\left(1-\frac{ \sqrt{\mnd} }{\mnd}  \right)\\ [0.2cm]
	 & \ge & 3\cdot 2^{n-1}\mnd^2  \left(1- \frac{\log(2\mnd)}{\sqrt{\mnd}} \right). 
	\end{array} 
%	\label{eq:ts4}
	\end{equation}
	%\hm{for the inequality $ 2^{d \mnd } \ge 2^{n } $: note that the last element is considered to have $ 2^d $ indices for $ \a $, but in fact it has $ d\mnd -n$. So there is no problem with this inequality. }  
From this and \eqref{eq:totalSumtd2}, \eqref{eq:sumF2}, (\ref{eq:tsV2}) and \eqref{eq:tsV3}, we obtain
	\begin{equation}
	\label{eq:Slax2L}
	    \sum_{\ax\in G }l\ax \ge 2^{n-1}\left(\lfloor {r^2 }/{2} \rfloor + 4 \mnd^2 \right) \left(1-  \frac{ 12 \mnd^{3/2}  \log(2\mnd)}{r^2+8{\mnd}^2} \right) .
	\end{equation}
	Plugging this into \eqref{eq:td} yields the lower bound in \eqref{eq:td2} for $ q\ge 100 $. 
%So \eqref{eq:td}  yields    
%$$ 
%\begin{array}{lll}
%2^{n-1} \left(nr+\lfloor {r^2 }/{2} \rfloor + 4 \mnd^2\right) \left(1- \frac{ 12 \mnd^{3/2} \log(2\mnd) }{2nr+r^2+8{\mnd}^2}  \right)
%& \le & \td(\rcr) \\ [0.2cm] 
%& \le & 2^{n-1} \left(nr+\lfloor {r^2 }/{2} \rfloor + 4 \mnd^2 \right).
%\end{array}
%$$

	\vspace{.1in}
	\textsf{Case 2:}~$ 1 \le \mnd <100 $. Suppose $ q\ge 2 $ first.  For any $ \ax\notin V $ with $ 0 \le x \le \mnd-2$, we have $ \leap_1\ax \le \mnd - x -1$  since $ y_s>x $. Moreover, for any fixed $ x $ with $ 0\le x\le \mnd-2 $, there are $ 2^n-2^{d(x+1)} $ elements $ \ax $ in $ G\setminus V $. This together with the fact that $ \sum_{i=0}^{k}  iz^i  \le 2kz^k$ for $z=2^d $ implies that 
	\[
		A_1 \ge \sum_{x=0}^{\mnd -2}\sum_{\ax \notin V} \left(  x + 2 \right) = \sum_{x=0}^{\mnd -2} (2^n-2^{d(x+1)}) ( x + 2 ) \ge  2^{n-1} (\mnd^2- \mnd) .  
	\] 
%	calculations: 
%	$$\sum_{x=0}^{\mnd -2} (2^n-2^{d(x+1)}) ( x + 2 ) =\sum_{x=0}^{\mnd -2}   2^n  ( x + 2 ) - \sum_{x=0}^{\mnd -2}  2^{d(x+1)}  ( x + 2 ) =$$
%	$$ 2^{n-1} (\mnd^2 +\mnd -1) -2^{-d}  \sum_{x=2}^{\mnd -2}  2^{d(x+2)}   ( x + 2 ).$$ 
%	Since $ 2^{-d} \sum_{x=2}^{\mnd -2}  2^{d(x+2)}  ( x + 2 ) =2^{-d} \sum_{x=4}^{\mnd }  2^{dx} x  \le 2^{-d} (2\mnd) 2^{n+d-1}  =2\mnd 2^{n-1} $, the above is at least $2^{n-1} (\mnd^2 -\mnd -1) $.
	 % 	 If $r$ is large, then the ratio of first sum over the first two sums  is equal to $o(1)$.  If $r$ and $\sd$ are not large and $n$ is large, then we use the approximation for later use.  
	This together with (\ref{eq:totalSumtd2}), \eqref{eq:sumF2}, \eqref{eq:tsV2} and \eqref{eq:tsV3}  yields    
	\begin{equation}
	\nonumber
		\sum_{\ax\in G} l\ax 	\ge 2^{n-1}  \lfloor {r^2 }/{2} \rfloor. 
	\end{equation} 
	On the other hand, if $ \mnd=1 $, then $ V=G $ and it can be verified that $ \sum_{\ax\in G }l\ax =2^{n-1}\rr $.  Hence,  for $ 1\le q<100 $, we have
	%a lower bound can be obtained by summing up (\ref{eq:tsV2}), (\ref{eq:tsV1}) and the above lower bound, all multiplied by $2$, adding $\sum_{\a \in \ZZ_2^n} (r/2) $ when $r$ is even.   
\begin{equation}
\label{eq:Slax2_nd}
	 \sum_{\ax\in G} l\ax \ge 2^{n-1} \left(\lfloor {r^2 }/{2} \rfloor + 4\mnd^2\right)\left(1- \frac{8\mnd^2}{r^2+8\mnd^2} \right)  .
\end{equation} 
Combining this with \eqref{eq:td}, we obtain the lower bound in \eqref{eq:td2} for $ 1\le  q<100$. 
%
%For given $ \epsilon>0 $, if $ \log\mnd/\sqrt{\mnd}/ <\epsilon $, then the lower bound for $ \td(\rcr) $ in \textsf{Case 1} is stronger. Otherwise if $ \sqrt{\mnd}/\log\mnd \le 1/\epsilon $, then $ \mnd^2/(nr+r^2+\mnd^2) = O((nr+r^2+\mnd^2)^{-1}\mnd^2) $ and so the lower bound for $\td(\rcr) $ in \textsf{Case 2} can be rewritten as $  2^{n-1} ( n r + \rr + 4 \mnd^2 )(1-O((nr+r^2)^{-1}))$. This completes the proof. 
\end{proof}

\begin{rem*}
{\em 
(a) When $ 2^nr$ is large, $ \alfndr $ is small and so \eqref{eq:td2} gives
$$
\td(\rcr) \approx 2^{n-1} \left( n r + \rr +4 \nd^2\right).
$$  

(b) Define
 \begin{equation}
 \label{eq:gndr}
 	 \btndr:= \left \{
 	 \begin{array}{ll}
	 	{(12\nd^{3/2}\log(2{\nd}))}/{( r^2+{8\nd^2})} , & \text{if } \nd \ge 100,\\ [.2cm]
 	   {8\nd^2}/{(r^2+8\nd^2)}, & \text{if } \nd < 100.
 	 \end{array} \right . 
 \end{equation}  
  Since $ r\ge 2n/d $,  $\btndr $ can be arbitrarily small for sufficiently large $r$. In the next section we will use the following bounds obtained from the proof of Theorem \ref{thm:td2}:
\begin{equation}
\label{eq:bnds4lax2}
2^{n-1} \left ( \rr + 4\nd^2\right ) (1-\btndr) \le  \sum_{\ax\in G} l\ax \le 2^{n-1} \left ( \rr +  4 \nd^2 \right ). 
\end{equation} 
The lower bound here is sharp when $ n=d $ (see Case 2 in the proof of Theorem \ref{thm:td2}), while the upper bound is nearly tight for sufficiently large $2^nr$.
%that is for $ \mnd=1 $ the lower bound gives a very tight value for the sum. and when $ r $ is very large, then the upper bound is sharp since the lower bound approaches to the upper bound.
}
\end{rem*}

\section{Forwarding indices}
\label{sec:indices}

%The edge-forwarding and vertex-forwarding indices of a graph were introduced in \cite{heydemann1997} and \cite{chung1987} respectively to measure the maximum loads on the edges and vertices of the graph. 
An all-to-all routing, abbreviated as a \emph{routing} in the sequel, {in} a connected graph is a set of oriented paths in the graph that contains exactly one path between every ordered pair of vertices. %The unique path in a routing from a vertex $u$ to another vertex $v$ is used to transmit data from $u$ to $v$.
A \emph{shortest path routing} is a routing which consists of shortest paths. 

Given a graph $X$ and a routing $R$ in $X$, the \emph{load} of an edge  $e \in E(X)$ with respect to $R$  is the number of paths in $R$ passing through  $e$ in either direction. The maximum load on edges of $X$ with respect to $R$ is denoted by $\pi(X,R)$.  The \emph{edge-forwarding index} \cite{heydemann1997} of $X$ is defined as
\begin{equation}
\label{eq:piDef}
	\pi(X) = \min_R \pi(X,R),
\end{equation}
where the minimum is taken over all routings {$R$ in} $X$. Similarly, the load of {$v \in V(X)$ with respect to} a routing $R$  is the number of paths in $R$ with $v$ as an internal vertex. The maximum load on vertices of $X$ with respect to $R$ is denoted by  $\xi(X,R)$. The \emph{vertex-forwarding index} \cite{chung1987} of $X$ is  defined as  
\begin{equation}
\label{eq:xiDef}
	\xi(X) = \min_R \xi(X,R),
\end{equation}
with the minimum taken over all routings $R$ in $X$. The \emph{minimal edge-} and \emph{vertex-forwarding indices} of $X$, denoted by $\pi_m(X)$ and $\xi_m(X)$, are defined {in the same way as in} (\ref{eq:piDef}) and (\ref{eq:xiDef}) respectively, with the minimum  taken over all shortest path routings {in} $X$.   These four forwarding index problems are known to be NP-complete {for general graphs} \cite{kosowski2009,saad1993}.

In this section, $\alfndr$ and $ \btndr $ are as given in \eqref{eq:fndr} and \eqref{eq:gndr}, respectively.  

\subsection{Vertex-forwarding index}
\label{sec:vfi}

It is known \cite{heydemann1989} that any Cayley graph $X$ admits a shortest path routing that loads all vertices uniformly. It follows that $\xi(X)=\xi_m(X)= \sum_{v\in V(X)} \dist(u,v) - (|V(X)|-1)$ \cite[Theorem 3.6]{heydemann1989}, where $u$ is any fixed vertex of $X$.

\begin{thm} 
\label{thm:xi} We have  $ \xi(\rcr)=\xi_m(\rcr) $ and the following hold:  
	\begin{enumerate}[\rm (a)]  
		\item  if $dr =n$ and $ r\ge 2^{9} $, then 
		$$2^{n-2}r^2(2d+ 5)\left (1 - \frac{4\log^2 r+4}{2n+5r}\right)\le \xi(\rcr) \le 2^{n-2}r^2(2d+ 5)\left (1-\frac{12r-8}{2nr+5r^2}\right );$$ 
		and  if $dr =n$ and $3\le r< 2^{9} $, then 
		$$2^{n-2}(2nr+ r^2-4r) \le \xi(\rcr) \le 2^{n-2}(2nr+ 5r^2-12r+8);$$
		\item if $ dr\ge 2n $, then 
		$$ 2^{n-1}\!  \left( n r\! +\! \rr \!+\!4 \nd^2\right)\! (1-\alfndr) \le \xi(\rcr)\! \le\!  2^{n-1}\!  \left( n r\! +\! \rr \! +\!4 \nd^2\right) .$$  
	\end{enumerate}
\end{thm} 

\begin{proof}
	Since $\xi(\rcr)=\xi_m(\rcr)= \td(\rcr)- 2^nr + 1$, the results follow from Theorems \ref{thm:td1} and \ref{thm:td2}.  
\end{proof} 

As a consequence of Theorem \ref{thm:xi} and Lemma \ref{lem:ccc}, we obtain the following corollary, of which part (c) gives the known result on the vertex-forwarding index of the cube-connected cycles, that is, $ \xi(CC_n)=\xi_m(CC_n)=7n^22^{n-1}(1-o(1)) $ \cite{yan2009}. 

\begin{cor}
	\begin{enumerate}[\rm (a)]
		\item $\xi(Q_n( d,n)) =  \xi_m(Q_n( d,n)) =  2^{n-1} (n^2+\lfloor n^2/2\rfloor + 4\nd^2 ) \; (1-O((\log n)\; /(\sqrt{\nd}(3d^2+8)))$ for $d\ge 2$; 
		\item if $ r\ge 2^{9} $, then $ \xi(COR(d,r))= \xi_m(COR(d,r)) $ and 
			$$2^{dr-2}r^2(2d+ 5)\left (1 - \frac{4\log^2 r}{(2d+5)r}\right)\le \xi(COR(d,r)) \le 2^{dr-2}r^2(2d+ 5)\left (1-\frac{12r-8}{(2d+5)r^2 }\right );$$ 
				and  if  $ 3\le r< 2^{9} $, then  $ \xi(COR(d,r))= \xi_m(COR(d,r)) $ and
				$$2^{dr-2}(2dr^2+ r^2-4r) \le \xi(COR(d,r)) \le 2^{dr-2}(2dr^2+ 5r^2-12r+8);$$
		\item  If $3\le n<2^9$, then $\xi(CC_n) = \xi_m(CC_n)$ and
		$$ 2^{n-2}(3n^2-4n)\le  \xi(CC_n) \le  2^{n-2}(7n^2-12n+8); $$
		and if $ n\ge  $, then $\xi(CC_n) = \xi_m(CC_n)$ and (\hspace{1sp}\cite{yan2009}) 
		$$ 7\cdot 2^{n-2}n^2 (1-(4\log^2 n)/(7n)) \le  \xi(CC_n) \le 7\cdot 2^{n-2}n^2 (1-(12n-8)/(7n^2)). $$ 
	\end{enumerate}
\end{cor}

\subsection{Edge-forwarding index}
\label{sec:efi}

We will use the theory of orbit proportional Cayley graphs \cite{shah2001} in our study the edge-forwarding index problem for $\rcr$. Given a graph $X =(V,E)$ and a subgroup $ H$ of $\Aut(X) $, the \emph{$H$-orbit} on $E(X)$ containing a given $e\in E(X) $ is $ \{g(e): g\in H\} $, and the \emph{stabiliser} of $ u \in V(X)$ in $ H $ is $H_u= \{g\in H: g(u)=u \} $. Define 
$$
H_{u,v} = (H_u)_v = \{g\in H: g(u)=u,g(v) =v \} 
$$ 
for distinct vertices $u,v \in V(X)$. 

Let $ \bar{R} = \cup_{(u,v)\in V \times V} \bar{R}_{uv}$ be the set of all paths in $X$, where $\bar{R}_{uv}$  {is the set} of all $uv $-paths  {in $X$}. A \textit{uniform flow} \cite{shah2001} in $X$ is a function $f: \bar{R}  \rightarrow [0,1]$ such that $\sum_{P \in \bar{R}_{uv}} f(P) =1$ for any distinct vertices $u, v \in V$. 	A path $ P $ in $ X $ is \emph{active} (under $ f $) \cite{shah2001} if $ f(P) >0$. The flow $f$ is called \emph{integral} if $f(P) \in \{0,1\}$ for any $P \in\bar{R} $. An integral uniform flow is essentially the same as an all-to-all routing. 
Given a subgroup $ H \le \Aut(X)$,  a uniform flow $f$ is called  {\emph{$H$-invariant}} if $f(P) = f(g(P))$ for all $ g \in H$ and $ P \in \bar{R}$, where $g(P)$ is the image of $P$ under $g$. 

\begin{lem}(\hspace{1sp}\cite[Theorem 1]{shah2001})
	\label{lem:intflow}
	Let $X = (V, E)$ be a graph and $H$ a subgroup of $\Aut(X)$. Then there exists an $H$-invariant uniform flow $f^*$ in $X$ such that any active path under $ f^*$ is a shortest path and the number of active paths is at most $ | H_{uv} | $.  
\end{lem}

The $H$-invariant uniform flow in Lemma \ref{lem:intflow} is integral if $ |H_{uv}| =1 $ for every pair $u, v$ of distinct vertices. Denote by $E_1,E_2,\dots,E_k$ the $H$-orbits on $E(X)$. Of course $\{E_1,E_2,\dots,E_k\}$ is a partition of $E(X)$. We say that $X$ is \emph{$H$-orbit proportional} \cite{shah2001} if for any shortest $uv$-path $P$ and any $uv$-path $P' $ in $X$,  
\begin{equation}
\label{eq:orbPr}
	|E(P) \cap E_i | \le | E(P') \cap E_i|,\quad i=1,2,\ldots, k .
\end{equation}
In particular, for $ 1\le i \le k $, $ |E(P) \cap E_i | = \min | E(P') \cap E_i| $ with the minimum running over all $ P' \in \bar{R}_{uv} $. Moreover, for $ 1\le i \le k $, we have $|E(P) \cap E_i | = | E(P') \cap E_i|$ if both $P$ and $P'$ are shortest $uv$-paths. Not all Cayley graphs are orbit proportional. 

It was proved in \cite[Theorem 4]{shahrokhi1995} that if $ X $ is $H$-orbit proportional and $f^* $ is an $H$-invariant uniform flow in $ X $ such that any active path is a shortest path (the existence of $ f^* $ is guaranteed by Lemma \ref{lem:intflow}), then  $\pi(X) = \max_{e\in E(X)} \sum_{P:e \in P} f^*(P)$.  
On the other hand, for $ e,e'\in E_i$, we have $\sum_{P:e \in P} f^*(P) = \sum_{P:e' \in P} f^*(P)$ since $ f^* $ is $H$-invariant. Therefore, what was proved in \cite[Lemma 5]{shahrokhi1995} is the following result:   
\begin{equation}
\label{eq:fP}
	\pi(X)=\pi_m(X) = \max_{e\in E(X)} \sum_{P:e \in P} f^*(P)= \max_{1\le i \le k} \dfrac{\sum_{(u,v)} |E(P_{uv}) \cap E_i| }{|E_i|},
\end{equation}
where $P_{uv}$ is any shortest $uv$-path in $ X $.

Now let us return to the edge-forwarding index problem for $\rcr$. Since $\rcr$ is vertex-transitive, by \cite{heydemann1989}, we have 
\begin{equation}
\label{eq:trivial}
\pi(\rcr) \ge |G| \ \td(\rcr)/ |E|.
\end{equation}
Define
\begin{align*}
	& E_0 := \{ \{\ax, (\a, x+1)\}: \ax \in G  \},\\ \nonumber
	& E_{i} := \{ \{\ax, (\a + \e_{i+dx}, x) \} : \ax\in G\}, \; 1\le i \le d.\nonumber
\end{align*}
Then $ |E_0| = 2^{n} r$ {and} $ |E_i| = 2^{n-1}r $ for $1 \le i \le d$. It can be verified that $\{E_0, E_{ 1},\dots,E_{ d}\}$ is a partition of the edge set of {$\rcr$}. 

Since $\rcr$ is a Cayley graph on $G$, $G$ can be viewed as a subgroup of $\Aut(\rcr)$ (see Section \ref{sec:term}). So we can talk about $G$-orbits on $ E(\rcr) $. 

\begin{lem}
\label{lem:edgeOrb}
	$ E_0, E_{1},\dots,E_{d} $ are {the $G$-orbits on $E(\rcr)$}.
\end{lem}

\begin{proof}
{Since $({\bf 0}_n, 0)$ is the identity element of $G$, by Definition \ref{def:rcr}, $E_0$ is the $G$-orbit on $E(\rcr)$ containing $\{({\bf 0}_n, 0), ({\bf 0}_n, 1)\}$ and $E_i$ is the $G$-orbit on $E(\rcr)$ containing $\{({\bf 0}_n, 0), ({\bf e}_i, 0)\}$, $1 \le i \le d$. Since $\{E_0, E_1, \dots, E_{ d}\}$ is a partition of $E(\rcr)$, these are all $G$-orbits on $E(\rcr)$.} 
%For any two ring edges $ e,e' \in E_0$ and $ g = (\c,z) \in  G $, the edge $g e $ belongs to $E_0 $ and there exists some $ g' \in \Gamma$ such that $ e' =g' e $.  The same claim is valid for $ E_i$, $i=1,\dots ,d$. 
%For any $ g  = (\c,z) \in  G $ and $ \{(\a,x),(\a,x \pm 1)\} \in E_0$, for some $ \a \in \ZZ_2^n,x\in \ZZ_r$, we have $g(\a,x) = (\c+\a M^{dz}, x+z )  $, $g(\a,x \pm 1) = (\c+\a M^{dz}, x+z \pm 1 )  $. That is $ \{g(\a,x),g(\a,x \pm 1)\} \in E_0$. Similarly, for any $1\le i \le d $ and $\{(\a,x),(\a+\e_t,x)\} \in E_i $, where $\a \in \ZZ_2^n,x\in \ZZ_r, t =dx+i \mod{n}$, we have $g(\a+\e_t,x) = (\c+(\a+\e_t) M^{dz}, x+z ) = ((\c+\a M^{dz}) +\e_{t'}, x+z ) $, where $t'= t +dz = d(x+z) + i \mod{n}$, which implies $\{g(\a,x),g(\a+\e_t,x)\} \in E_i $. 
\end{proof}

\begin{lem} \label{lem:orbPrp}
	$\rcr$ is $G$-orbit proportional if and only if $ n\equiv0\mod{d} $.
\end{lem}

\begin{proof}
Suppose that $ n\equiv0\mod{d} $. Let  $\ax \in G$ and let $P $ and $P' $ be paths from $ \oo $ to $ \ax $ in $\rcr$. Suppose that $P$ is a shortest path  so that $ |P|=\|\a\| +l\ax $ by Lemma \ref{lem:axSeq}. For $1\le j \le n$, if $ a_j =1 $ (respectively, $ a_j =0 $), then $P'$ contains an odd (respectively, even) number of cube edges $\{(\b,y),(\b+\e_j,y)\} $ in direction $\e_j$, and $P$ contains exactly one (respectively, zero) such edges by Lemma \ref{lem:axSeq}.  On the other hand, we claim that  any two  cube edges  $\{(\b,y),(\b+\e_j,y)\} $ and  $\{(\b',y'),(\b'+\e_j,y')\}$ in the same direction $\e_j$ are in the same  $G$-orbit $E_k$ for some $k$. In fact, by Definition \ref{def:rcr}, we have  $ j \in D(y) \cap D({y'}) $ and so $k + dy \equiv k' + dy' \mod{n} $ for some $1\le k,k' \le d$. Since $ n \equiv 0 \mod{d} $ by our assumption, $  k = k' $ and hence $(\b+\e_j,y) = \by (\e_k, 0)$ and $ (\b'+\e_j,y') = (\b',y')(\e_k, 0)$. In other words, both $\{(\b,y),(\b+\e_j,y)\}$ and  $\{(\b',y'),(\b'+\e_j,y')\}$ are in the $G$-orbit $E_k$. This together with what we proved above implies that $ | E(P) \cap E_i | \le | E(P') \cap E_i| $ for $1 \le i\le d$. Moreover, $|E(P) \cap E_0| \le |E(P') \cap E_0|$, for otherwise there exists an $\ax $-sequence $ \hat{x} $ obtained from segments of $ P' $ such that $ l(\hat{x}) < l\ax$, which is a contradiction. Therefore, (\ref{eq:orbPr}) is satisfied and so $\rcr$ is $G$-orbit proportional.  
	
Suppose that $ n\not\equiv0\mod{d} $. Then $ dr\ge 2n $ by \eqref{eq:grp}.  By Lemma \ref{lem:rpit}, there exist $ 1\le j\le n $, $ 1\le k,k'\le d $ and $ y,y'\in \ZZ_r $ such that $ y\ne y' $, $ k \ne k' $ and $ j = dy + k = dy'+k' \mod{n} $. Let $ y $, $ 0\le y <r$, be such that $ j\in D(y) $ and $ \min\{ y,r-y \} \le \min\{ y',r-y' \} $ for any $ 0\le y'<r $ with $ j\in D(y') $. The path $(\bo,0), \ldots, (\bo,y), (\e_j,y),\ldots,(\e_j,0) $ is a shortest path with exactly  one  edge in $ E_k $, and the path $ (\bo,0), \ldots, (\bo,y'), (\e_j,y'), \ldots, (\e_j,0) $ has exactly  one  edge in $ E_{k'} $. Since the first path has an edge in $ E_k $ while the second path  does not  contain  any edge in $ E_k $, (\ref{eq:orbPr}) is not satisfied by these paths and $E_k$. Hence $\rcr$ is not $G$-orbit proportional. 
\end{proof}

In view of the discussion at the beginning of Section \ref{sec:shrtpth}, any set $\{P_{\ax} : \ax \in G\}$ of shortest paths in $\rcr$ {starting from $\oo$} gives rise to a shortest path routing in $\rcr$ defined by 
\begin{equation}
\label{eq:R}
\{ \by P_{\ax}  : \ax,\by \in G\}.
\end{equation}

\begin{thm}\label{thm:pi}  
	\begin{enumerate}[\rm (a)]
		\item Suppose $dr =n $. We have $ \pi(\rcr) =\pi_m(\rcr)  $ and the following hold: 
		\begin{enumerate}[\rm (i)]
			\item if $ 3\le r \le 6 $, then
			 \[
			 	\pi(\rcr)=2^{n} r^2 ;
			 \]
		 	\item if $ 7 \le r < 2^{9}$,  then 
		 	\begin{equation}
			 \nonumber
			 2^{n} r^2\le \pi(\rcr) \le  2^{n-2} (5r^2-8r+8); 
			 \end{equation}
			 \item if $r \ge 2^{9} $, then
			\begin{equation}
				\nonumber
				2^{n-2} r^2\max\left \{4,5(1 - (4 \log^2 r/r))\right \} \le \pi(\rcr) \le   2^{n-2} (5r^2-8r+8) .
			\end{equation} 
		\end{enumerate}
		\item Suppose $dr \ge 2 n $ and $ n\equiv 0\mod{d} $.  Then  $ \pi(\rcr) =\pi_m(\rcr)  $ and 
		\begin{align}
		\label{eq:pi3}
			2^{n-1}\max\left\{ 2nr/d,\left(\rr +4 n^2/d^2 \right)  (1-\btndr) \right \}& \\  \nonumber 		
			 \le \pi(\rcr)    &\le 2^{n-1}\left ( \rr +  4 n^2/d^2  \right).
		\end{align} 		 
		\item If $ dr \ge 2n $, then 
		\begin{align} 
		\label{eq:pi4}	  
			\dfrac{2^n}{d+2} (nr +\rr + 4\nd^2) (1-\alfndr) \le & \pi(\rcr) \le \\ 
			\nonumber
			   \pi_m(\rcr) \le & 2^{n-1}  \max  \left\{ {4nr}/{d} + 2r, \rr +  4\nd^2 \right\}.  
		\end{align} 
	\end{enumerate}
\end{thm}	 
 
\begin{proof}		
Let $ P_{\ax} $ be a shortest path in $\rcr$ from $ \oo $ to $ \ax $. Since $G$ is regular on $G$ in its left-regular multiplication, we have $G_{u} = \{\oo\}$ and so $G_{uv} = \{\oo\}$ for any two distinct vertices $u $ and $ v $ of $ \rcr$. Thus, by Lemma \ref{lem:intflow}, there is exactly one $uv$-path $P_{uv}$ such that $f^*(P_{uv})=1$ and $P_{uv}$ is a shortest path. 
	
If $ n\equiv 0\mod{d} $, then by Lemma \ref{lem:orbPrp}, $\pi(\rcr)$ is given by (\ref{eq:fP}). Since  $ \sum_{(u,v) \in G \times G} |E(P_{uv})\cap E_i| = 2^nr\sum_{\ax \in G } |E(P_{\ax}) \cap E_i| $, $ 0\le i\le d $, we have 
\begin{align}
\nonumber
	\pi(\rcr)\! =\!\pi_m(\rcr)\! &= \max_{0 \le i \le d} \frac{ 2^nr \sum_{\ax \in G}|E(P_{\ax})\cap E_i|}{|E_i|} \\
	\nonumber
	&= \max\left\{ \!\!\sum_{\ax \in G}\!|E(P_{\ax})\!\cap\! E_0|, \max_{ 1 \le i \le d } 2 \!\!\sum_{\ax \in G}\!|E(P_{\ax})\!\cap\! E_i| \right\}\\
		\label{eq:pif}
	& = \max\left\{ \sum_{\ax \in G}l\ax , \max_{ 1 \le i \le d } 2r \sum_{\a \in \ZZ_2^n}|E(P_{(\a,0)}) \cap  E_i|\right\}.
\end{align}  

(a)~Suppose $dr = n$. For $ 1 \le i \le d $ and $ 0 \le y \le r-1 $, $ P_{\ax} $ contains exactly one edge in $ E_i $  if and only if $ a_{d y+i} =1 $. Hence $ \sum_{\a \in \ZZ_2^n}| E(P_{(\a,0)}) \cap  E_i| = 2^{n-r} \sum_{y=0}^r  y \binom{r}{y} = 2^{n-1}r $. 
This together with (\ref{eq:Slax1UB}),  (\ref{eq:Slax1R}), \eqref{eq:Slax1b} and (\ref{eq:pif}) yields the result.  	
%On the other hand, for $r\ge 2^{9}$, $5\cdot 2^{n-2}r^2(1 - 4 (\log^2 r)/r ) \le \sum_{\ax\in G} l\ax \le  2^{n-2}(5r^2-8(r-1))$ by (\ref{eq:Slax1UB}) and  (\ref{eq:Slax1R}), and for $r < 2^{9}$,  $2^{n-2}r^2 \le \sum_{\ax\in G} l\ax \le  2^{n-2}(5r^2-8(r-1)) $ by \eqref{eq:Slax1UB}. Therefore, by (\ref{eq:pif}) the result is true when $ dr =n $. 

% % % % % % % % % % % %	
(b)~Suppose $ dr \ge 2n $ and $ n\equiv 0\mod{d} $. For $ 1 \le i \le d $,  $P_{\ax} $ has exactly one edge in $ E_i$ if and only if $ a_{dy+i} =1 $ for some $ 0\le y \le r-1 $. Since by Lemma \ref{lem:rpit} the number of distinct values of $ dy+i $ is $ n/d $, we have $ \sum_{\a \in \ZZ_2^n}| E(P_{(\a,0)}) \cap  E_i | = 2^{n-n/d} \sum_{y=0}^{n/d}  y \binom{n/d}{y} = 2^{n-1}n/d$.  Therefore, by \eqref{eq:bnds4lax2} and (\ref{eq:pif}) we get  \eqref{eq:pi3} immediately.  
	
(c)~Suppose $ n \not\equiv 0 \mod{d} $.  Denote by $ R $ the routing (\ref{eq:R}) based on our chosen shortest paths $P_{\ax}$. Let  $ 0\le x<r $ and $ i\in D(x) $ be fixed. For any pair of vertices $ (\by,\cz) $, the path $ P\in R $ from $ \by $ to $ \cz $ passes through $ e=\{ (\a,x ),(\a+\e_i,x)\} $ for some $ \a \in \ZZ_2^n $  provided that its corresponding optimal sequence (obtained from \eqref{eq:xHt})  contains $ x $.  By the  construction of  the optimal sequence for $ P $,  we have $ r-\nd+1 \le x-y \le \nd -1 $ and $ b_i\ne c_i $. Therefore, there are at most $ 2^{2n-1 }(2\nd - 1) r  \le 2^{2n}nr/d + 2^{2n-1 }r $ paths in $ R $ containing $ e $.  
On the other hand, for any path $ P: \by,\ldots,\ax,(\a+\e_i,x),\ldots,\cz $ in $ R $ that passes through $ e $ and any $ \a'\in\ZZ_2^n $, the path $gP: g\by,\ldots,(\a',x),(\a'+\e_i,x),\ldots,g\cz $ in $R$, where $ g = (\a'-\a,0) $, passes through the edge $ e'= \{(\a',x ),(\a'+\e_i,x)\} $. Therefore, the  paths in $ R $ uniformly load the edges of $\{ \{(\a,x),(\a+\e_i,x)\}: \a\in \ZZ_2^n\}$. Thus, the load on each cube edge $ \{ (\a,x ),(\a+\e_i,x)\} $ under $R$ is at most $ 2^{n+1}nr/d + 2^{n}r $.  
	
On the other hand, the load on ring edges of $ \rcr $ under $ R $ is uniform since $ R $ is $ G $-invariant and the set of ring edges forms a $ G $-edge orbit. Similar to (b), the load of $ R $ on ring edges is given by \eqref{eq:bnds4lax2}. This together with the upper bound for the load on the cube edges under $ R $ gives the upper bound in \eqref{eq:pi4} for $ \pi(\rcr) $ and $ \pi_m(\rcr) $. The lower bound in \eqref{eq:pi4} follows from \eqref{eq:trivial} and Theorem \ref{thm:td2}.  
\end{proof}

\begin{rem*}
\em{
(a) By Theorem \ref{thm:pi}, when $ r $ is large, we have 
\[
	\pi(\rcr) \approx 5\cdot 2^{n-2} r^2,\;\; \mbox{if $ dr=n $}
\]
\[
\pi(\rcr) \approx  2^{n-1}\left (\rr +  4 n^2/d^2 \right), \;\; \mbox{if $ dr\ge 2n $ and $ n\equiv 0\mod{d} $}.
\]	
 
(b) The edge-forwarding index of any graph is bounded from below \cite{heydemann1989} by the sum of the distances between all order pairs of vertices divided by the number of edges. Obviously, a graph with edge-forwarding index close to this bound has relatively small bottleneck-congestion on edges as far as routings are concerned. In the case of $\rcr$, this trivial lower bound gives $\pi(\rcr) \ge 2\td(\rcr)/(d+2)$. Using Theorems \ref{thm:td1}, \ref{thm:td2} and \ref{thm:pi}, one can verify that the ratio of $ \pi(\rcr)$ to this trivial lower bound is no more than (i) $ (5d+10)/(4d+2)$ if $ dr=n $ and $7 \le r < 2^{9} $, (ii) $ 1.82 (5d+10)/ (4d+10) \le 1.95$ if $ dr=n $ and $ r\ge 2^{9} $, (iii) $1/(1-\alpha_{n,d,r})$ if $ n\equiv0\mod{d} $ and $ dr =kn $ for some integer $ k\ge 3$, and (iv) $ \max\{2(d+2)(2n+d)/d(2n+r), 1 + d (k^2+8)/((2 dk+  k^2+8 )(1-\alfndr))\} < 6 $ if $ dr=kn $ for some integer $ k\ge2 $. In the first three cases $\rcr$ has relatively small edge-forwarding index. 
	}
	%$(2k^2 +(8+k)d +16 )/ \left ( (2k^2 + 4kd +16)(1-\alpha_{n,d,r})\right ) \le 1/(1-\alpha_{n,d,r})$ if $ n\equiv0\mod{d} $ and $ dr =kn $ for some integer $ k\ge 3$,
\end{rem*}

 By Lemma \ref{lem:ccc}, we obtain the following corollary of Theorem \ref{thm:pi}. In particular, we recover the formula  $ \pi(CC_n)=5n^22^{n-2}(1-o(1)) $ (\hspace{1sp}\cite[Theorem 3]{shah2001}). In fact, we give more accurate lower and upper bounds for $ \pi(CC_n) $ and $ \pi_m(CC_n) $. We observe that the $d$-ply cube-connected cycles $Q_{kd}(d,kd)$ with $ k, d \ge 2 $ have smaller edge-forwarding indices than the usual cube-connected cycles.  

\begin{cor}
	\begin{enumerate}[\rm (a)]
		\item $  2^{n-2}n^2\max\left \{ 4,\; 5 (1 - 4 \;(\log^2 n)/n )\right \} \le  \pi(CC_n) = \pi_m(CC_n)  \le   5\cdot 2^{n-2}n^2(1-(8n-8)/(5n^2))$ if $ n\ge2^{9}$;  $2^{n} n^2 \le \pi(CC_n) = \pi_m(CC_n) \le 5\cdot 2^{n-2} n^2 (1-(8n-8)/(5n^2))$ if $ 7\le n < 2^{9}$; and $ \pi(CC_n) = \pi_m(CC_n) = 2^{n} n^2 $ if $ 3\le n\le 6 $. 
		\item $ 2^{dr-2} r^2\max\left \{4,5(1 - 4 (\log^2 r)/r )\right \} \le \pi(COR(d,r)) = \pi_m(COR(d,r)) \le   2^{dr-2} (5r^2-8(r-1)) $ if $ r\ge2^{9}$;  $2^{dr } r^2 \le \pi(COR(d,r))= \pi_m(COR(d,r)) \le  2^{dr-2} (5r^2-8(r-1)) $ if $ 7 \le r < 2^{9}$; and $ \pi(COR(d,r))= \pi_m(COR(d,r))= 2^{dr}  r^2  $ if $ 3\le r\le 6 $.
		\item 	$ 2^{kd-1}\max\left\{ 2k^2d ,\left(\lfloor k^2d^2/2\rfloor +4 k^2 \right) (1-\beta_{kd,d,kd}) \right \}  \le \pi(Q_{kd}(d,kd))  =\pi_m(Q_{kd}(d,kd))   \le 2^{kd-1}  \max\left \{   2k^2d ,  \lfloor k^2d^2/2\rfloor +  4 k^2  \right \} $ for any integers $ d\ge 2 $ and $ k\ge 1 $ with $ kd\ge 3 $. 
	\end{enumerate}
\end{cor}

\section{Bisection width}
\label{sec:bw}

Given a graph $X$ and a subset $U$ of  $V(X)$, let $\delta(U,\overline{U})$ be the subset of $E(X)$ consisting of those edges with one end-vertex in $U$ and the other in $\overline{U}:= V(X) \setminus U$.  A \emph{bisection} of $X$ is a partition  $\{U  ,\overline{U} \}$ of $V(X)$ such that  $|U|$ and $|\overline{U}|$ differ by at most one. The \emph{bisection width} of  $X$, denoted by $\bsw(X)$, is the minimum of $|\delta(U,\overline{U})|$ over all bisections $\{U  ,\overline{U} \}$ of $X$. It is known \cite{mans2004} that the decision problem for $\bsw(X)$ is NP-complete for general graphs $X$. 

Let $R^*$ be a routing in $X$ such that $\pi(X)=\pi(X, R^*)$ and $\{U, \overline{U} \}$ a partition of $V(X)$. Then the total load  on the edges of $\delta(U,\overline{U})$ under $R^*$ is at most $\pi(X)|\delta(U,\overline{U})|$. On the other hand, there are exactly $2 |U| |\overline{U}|$ paths in $R^*$ with one end-vertex in $U$ and the other in $\overline{U}$. Therefore, 
\[
\pi(X)|\delta(U,\overline{U})| \geq 2 |U| |\overline{U}|.
\]
In particular, for a bisection $\{U  ,\overline{U} \}$ of $X$ with $ \delta(U,\overline{U})=\bsw(X) $, this yields  \cite{mohar1997}
\begin{equation}
 \label{eq:lbew}
	\bsw(X) \ge \frac{ 2\lfloor |V(X)| /2\rfloor \lceil  |V(X)| /2 \rceil }{\pi(X) } = \frac{\lfloor |V(X)|^2 / 2 \rfloor }{\pi(X)}.
\end{equation}

It was claimed in \cite[Corollary 1]{sun2000} and \cite[Theorem 2]{hu2005} (with a difference for case $ r=1 $) that the bisection width of $Q_n^-(d,r)$ is equal to $ 2^{n-1} dr/n $  and shown in \cite[Theorem 12]{xie2013} that $\bsw(Q_n^-(d,r))\le 2^{n-1} dr/n$. However, it has been unknown whether the upper bound on $ \bsw(Q_n^-(d,r)) $ in \cite{xie2013} is sharp or not.  
Using Theorem \ref{thm:pi}, we give sharp bounds on $\bsw(\rcr)$ in the following theorem. Our result shows in particular that in some cases $\bsw(Q_n^-(d,r)) < 2^{n-1} dr/n$ and so the related results in \cite{sun2000, hu2005} are incorrect and the upper bound on $ \bsw(Q_n^-(d,r)) $ in \cite{xie2013} is not sharp in general.  
 
\begin{thm}\label{thm:ew}
	\begin{enumerate}[\rm (a)]
		\item  If $dr =n $, then
		\begin{enumerate}[\rm (i)]
		 \item  $\bsw(\rcr) = 2^{n-1} $ when $ 3\le r\le 6$;
		 \item  $2^{n+1}/(5-(8(r-1)/r^2)) \;\le  \bsw(\rcr) \le 2^{n-1} $ when $r\ge 7$.
		\end{enumerate}
		\item  If $ dr\ge 2n $, then
		 \begin{enumerate}[\rm (i)]
		 	\item  $  2^{n+1} r^2/(r^2  +  8n^2/d^2)  \le \bsw(\rcr) \le 2^{n} \min \left\{dr/2n, 2 \right\}$ when either $r$ is even and $ n=kd $ for some $k\ge 2$; %We cannot say "either $r$ is even OR $ n=kd $ for some $k\ge 2$" since it contains cases $ n\not\equiv 0\mod d $, for which lower bound is not proved. 
		 	\item $ 2^{n+1} r^2/( r^2 + 8 ) \le \bsw(\rcr) \le 2^{n-1} \min \left\{ r,5 \right\}$ when $r$ is odd and $n = d $;
		 	\item $ 2^{n} \min\left\{dr/(4n+2d), 2r^2/(r^2\! +\! (8\nd^2)) \right\} \le  \bsw(\rcr) \le 2^{n}\! \min \left\{dr/2n, 2 \right\} $ when $ n\not\equiv0\mod{d} $.
	 	\end{enumerate}
	\end{enumerate}
\end{thm} 

\begin{proof}
	The lower bound in each case is obtained from (\ref{eq:lbew}) and the corresponding upper bound for $ \pi(\rcr) $ in Theorem \ref{thm:pi}. Hence we have 
	\begin{equation}
	\label{eq:rcrbswLb}
	\bsw(\rcr) \!\ge\! \left\{\!\!
		\begin{array}{l l}
		2^{n-1}, &\text{ if } dr =n \text{ and  }  3\le r\le 6, \\[.2cm]  
		2^{n+1} /(5-(8(r-1)/r^2)), &\text{ if } dr =n \text{ and  }  r\ge 7, \\[.2cm] 
		2^{n+1} r^2/(r^2\! +\!8n^2/d^2) , &\text{ if } dr \!\ge\! 2n \text{ and  }  n \equiv 0 \!\!\!\mod{d}, \\[.2cm] 
		2^{n} \min\left\{dr/(4n+2d), 2r^2\!/\!(r^2\! +\!\! 8\nd^2) \right\}\!,\! &\text{ if } dr\! \ge \! 2n  .
		\end{array} \right .  
	\end{equation}
	 %If $ dr =n $ and $ r\ge 36 $, then $ \pi(\rcr )  \le  2^{n-2}(5 r^2-7r) $ and so $ \bsw(\rcr) \ge   2^{n+1} /(5-7/r)$. If $ dr =n $ and $3\le r\le35 $, then $ \pi(\rcr )  =  2^{n}r^2$ and so $ \bsw(\rcr) \ge   2^{n-1}$.	If $ dr \ge 2n $ and $ n \equiv 0 \mod{d} $, then $ \pi(\rcr) \le 2^{n-1}\max\left\{ {2nr}/{d} , r^2/2 +  4n^2/d^2 \right\} $ by  Theorem \ref{thm:pi}, and hence, by (\ref{eq:lbew}), $\bsw(\rcr) \ge 2^n \min\left\{dr/2n, 2 r^2/(r^2 + (8n^2/d^2)) \right\}$. If $ dr \ge 2n $ (without requiring $ n \equiv 0 \mod{d} $), then, by  Theorem \ref{thm:pi}, $ \pi(\rcr) \le 2^{n-2}\max\{ (8nr+4dr)/d,  r^2 + 8\nd^2  \} $ and thus, by (\ref{eq:lbew}), $\bsw(\rcr) \ge 2^{n} \min\left\{dr/(4n+2d), 2r^2/(r^2 + (8\nd^2)) \right\}$. 
		
	It remains to prove the upper bounds for $ \bsw(\rcr) $. Let $U$ be the set of vertices $ \ax $ of $ \rcr $ with $\a=(a_{1},a_2 ,\dots, a_{n-1},1)$. Then $\overline{U}$  is the set of vertices $ \ax $ such that $\a=(a_1,a_2,\ldots,a_{n-1} ,0 )$ and $\{ U , \overline{U}\}$ is a bisection of $ \rcr $. There is an edge between $\ax \in U$ and  $(\a +\e_n , x) \in \overline{U}$ if and only if $ n\in D(x) $. Thus,   
		$$ \delta(U,\overline { U }) =\{  \{(\a, x ) , (\a+\e_n, x ) \} : n\in D(x), 0\le x\le r-1\}.  $$
	 By Lemma \ref{lem:rpit}, there are $ d r/n$ distinct elements $ x $ such that  $n\in D(x) $.  Hence   $|\delta( U, \overline{U})| =  2^{n-1} d r  / n$  and so 
	\begin{equation}
	\label{eq:ubbw}
		  \bsw(\rcr) \le  2^{n-1} d r  / n.
	\end{equation}
	
	We now use other bisections to obtain better upper bounds in some cases.  
	
	\medskip 
	\textsf{Case 1:}~$ r $ even. Let $U= \{ \ax : \a \in \ZZ_2^n, 0\le x \le  r/2  - 1 \}$. Then $\{ U , \overline{U}\}$ is a bisection of $\rcr$ and  %Note that there is no cube edge  in $\delta(U,\overline  U })$.
		$$\delta(U,\overline { U }) = \left\{  \{(\a, x  ),   (\a , x+1 )\} :  \a \in \ZZ^n_2 , x = r/2 -1,r-1 \right\}.$$
	Hence $ \bsw(\rcr) \le   |\delta( U, \overline{U})| =  2^{n + 1}$. 
	
	\medskip 
	\textsf{Case 2:}~$ r $ odd. Let $U= \{\ax: \a \in \ZZ_2^n , 1 \le x \le (r-1)/2 \} \cup \{(\a,0): \a \in \ZZ_2^n, a_{n}=0  \}$. Then    $\overline{U}= \{\ax: \a \in \ZZ_2^n , (r+1)/2 \le x \le r-1  \} \cup \{(\a,0): \a \in \ZZ_2^n, a_{n}=1  \}$ and $\{ U , \overline{U}\}$ is a bisection of $\rcr$.  Two vertices $ ((a_1, \ldots, a_{n-1} ,1),0) $ and $((b_1,\ldots, b_{n-1},0),0) $ are adjacent if and only if  $ n\in D(0) $. Hence, if $n > d$, then these two vertices are not adjacent since  $n\not\in D(0)$. Thus, if $n > d$, then 
	\begin{align}
	\nonumber
		\delta(U,\overline { U }) = \{ & \{(\a, (r-1)/2  ), (\a ,(r+1)/2 )\} : \a \in \ZZ^n_2 \} \cup \\		\nonumber
		& \{ \{(\a, 0  ),   (\a ,1 )\} :   \a \in \ZZ^n_2,a_n=1 \} \cup \{ \{(\a, 0 ), (\a ,r-1 )\} : \a \in \ZZ^n_2,a_n=0 \}  .
	 	\end{align}
	Therefore, $|\delta( U, \overline{U})| =  2^{n + 1}$. If $n = d$, then $((a_1,\ldots, a_{n-1},1),0) $ and  $((b_1,\ldots, b_{n-1},0),0) $  are adjacent and so $\bsw(\rcr) \le |\delta( U, \overline{U})| = 2^{n+1}+2^{n-1}$.  
	
	Combining  \eqref{eq:ubbw} and the upper bounds for $ \bsw(\rcr) $ in Cases 1 and 2, we obtain 
	\[
	 \bsw(\rcr) \le \left\{
	 \begin{array}{l l}
		 \min\{ 2^{n-1}dr/n, 2^{n + 1} \} , &\text{ if } r \text{ even or } n>d , \\[.2cm]  
	  	\min\{ 2^{n-1}dr/n, 5\cdot 2^{n-1}\} ,  &\text{ if } r \text{ odd and }   n=d. \\ 
	 \end{array} \right . 
	 \]	
This together with \eqref{eq:rcrbswLb} completes the proof. 
\end{proof}

\appendix
 
\section{Proof of Theorem \ref{thm:rcriicay}}
\label{app:prfrcrii}

 In a graph two vertices are \emph{$t$-neighbours} of each other if the distance between them is equal to $t$. Obviously, in a vertex-transitive graph any two vertices have the same number of $t$-neighbours for any integer $t \ge 1$.

  	If $d r  \equiv 0\mod{n}$, then $Q_n^-(d,r) \cong \rcr$ by Lemma \ref{lem:isom}. So $Q_n^-(d,r)$ is a connected Cayley graph by Lemmas \ref{lem:basic} and \ref{lem:isom}. 
 	
 	In the rest of the proof we assume that $Q_n^-(d,r)$ is connected but $d r \not \equiv 0\mod{n}$. We will prove that $Q_n^-(d,r)$ is not vertex-transitive and hence not a Cayley graph. We achieve this by showing that the numbers of $t$-neighbours of $(\bo,r-1)$ and $(\bo, r') $ in $Q_n^-(d,r)$ are distinct for some $t$ to be defined later and 
 	\[
 		r' = \lfloor (r-1)/2 \rfloor.
 	\]
 	Note that $r'\ge 1$ as $r\ge3$.  Set
 	\[
 	D^{-}_x := D(-x)= \{ i-dx \mod{n}: 1\le i \le d\},\;\, x \in \ZZ_r, 
 	\] 
 	with the understanding that the elements of $F^{-}_x$ are between $1$ and $n$. Obviously, all these sets have size $d$. Note that $ (\a,x) $ and $ (\a+\e_j,x) $ are adjacent in $Q_n^-(d,r)$ if and only if $j\in D^{-}_x $.  We partition the set of $t$-neighbours $(\a_t,x_t) $ ($0 \le x_t \le r-1$) of $ (\bo,r-1) $ into $N_{t-2}, N_{t-1}$ and $N_t$, according to whether $\min\{ r-x_t-1,x_t+1\}$ is at most $ t-2 $, exactly $t-1$, and exactly $t$, respectively. Similarly, we partition the set of $t$-neighbours $(\b_t,y_t)$ ($0 \le y_t \le r-1$) of $(\bo,r')$ into $N'_{t-2}, N'_{t-1}$ and $N'_t$ according to whether $|y_t -r'|$ is at most $ t-2$, exactly $t-1$, and exactly $t$, respectively. We will prove that $|N_{t-2}| \le |N'_{t-2}|$, $|N_{t-1}| < |N'_{t-1}|$ and $|N_{t}| = |N'_{t}|$, and therefore $(\bo,r-1)$ and $(\bo, r') $ have different numbers of $t$-neighbours as required. 

 	Since $d r \not \equiv 0\mod{n}$, we can write $dr = an - q$, where $a$ and $q$ are integers with $1\le q \le n-1$ and $a\ge 2$. So  
\begin{equation}
\label{eq:f-1}
	D^{-}_{r-1} = q + D^{-}_{-1} := \{ q + d +i  \mod{n}: 1\le i \le d\}.
\end{equation}
Note that  the sets $ D^{-}_{i} $ and $ D^{-}_{i+1} $ have consecutive elements (modulo $ n $) and $|D^{-}_{-1} \cap D^{-}_{i} | =  | D^{-}_{r'} \cap D^{-}_{r'+i+1}|  $, for every $ 0 \le i < r' $. 
 	
\vspace{.05in}
\textsf{Claim 1:}~There exists $k$ with  $0 \le k \le r'- 1$ such that $ |D^{-}_{r' } \cup D^{-}_{r'+k+1}| \ne |D^{-}_{r-1} \cup D^{-}_{k} |$.
\vspace{-.1in}
	\begin{proof} 
 		First, we show that $D^{-}_{r-1} \cap D^{-}_{k}\ne \emptyset$ for some  $0 \le k <r'$. In fact, if $ q > n-2d$, then $D^{-}_{ r-1} \cap  D^{-}_0  \ne \emptyset $ since $n-d < q+d  < n+d $ implies $ q+d+i \mod{n} \in D^{-}_0  $ for some $ 1\le i \le d $.   
 		Assume $1 \le q \le n-2d$. Then $dr = an - q \ge an - (n-2d) = (a-1)n + 2d \ge (a-1)(2d+1)+2d = 2ad + (a-1)$, that is, $r \ge 2a + (a-1)/d$. Since $a \ge 2$, it follows that $r \ge 5$ and $r'\ge 2$. Let $\epsilon = r-1-2r'$. Then $d r' = (dr - (\epsilon+1)d)/2 $, and $\epsilon = 0, 1$ depending on whether $r$ is odd or even. Note that  $\cup_{x=0}^{r'-1} D^{-}_x =\{ 1,\dots, d \} \cup \{ n, n-1,\dots, 1-d (r'-1) \mod{n} \} $. Thus, if $D^{-}_{ r-1} \cap (\cup_{x=0}^{r'-1} D^{-}_x ) = \emptyset$, then the elements of $ D^{-}_{ r-1}$  are between $d$ and $1 - d(r'-1) \mod{n} $, that is, $ d < q+d +i < n + 1 - d(r'-1) = (2-a)n/2 + d + 1 + (q+(1+\epsilon)d)/2 $  for $1\le i \le d$. In particular, when $i=d$, this inequality yields $ (q+d)/2 < (2-a)n/2 + \epsilon d/2 +1$. However, this is impossible, because when $r$ is even and $a=2$, we have $q = an - dr$ is even and so $q\ge 2$, and in the remaining case we have $\epsilon = 0$ or $a > 2$. This contradiction shows that there exists some $0 \le k \le r'-1$ such that $D^{-}_{r-1}  \cap D^{-}_{k} \ne \emptyset$. 
 		
 		Now we show that for this particular $k$ we have $|D^{-}_{-1 } \cap D^{-}_{k}| \ne |D^{-}_{r-1} \cap D^{-}_{k} |$ or $|D^{-}_{-1 } \cap D^{-}_{k-1}| \ne |D^{-}_{r-1} \cap D^{-}_{k-1} |$. If $k = 0$, then $|D^{-}_{-1 } \cap D^{-}_{0}| \ne |D^{-}_{r-1} \cap D^{-}_{0} |$ since $ D^{-}_{-1} \cap D^{-}_{0} = \emptyset $. Assume that $1 \le k \le r'-1$ and $  |D^{-}_{-1} \cap D^{-}_{k} | =  | D^{-}_{r-1} \cap D^{-}_{k}| =s $ for some $s$ with $ 0 < s <d $ (note that $ s<d $ since $  D^{-}_{r-1} \ne D^{-}_{-1} $).  Since $ D^{-}_{k}  \cup D^{-}_{k-1}  = \{ 1-dk,\dots,2d-dk  \mod{n}\} $ and $ 2d \le n $, we have $ | D^{-}_{k} \cup  D^{-}_{k-1}| =2d$. So one of $ D^{-}_{-1} $ and $ D^{-}_{r-1} $ has $ d-s $ elements in $ D^{-}_{k-1} $, and the other has $ d-s $ elements in $ D^{-}_{k+1} $. Since $ D^{-}_{r-1} = q + D^{-}_{-1} $ for $ 1 \le q <n $, it follows that $  |D^{-}_{-1} \cap D^{-}_{k-1} | \ne  | D^{-}_{r-1} \cap D^{-}_{k-1}| $. 
 		Since $ |D^{-}_x|=d $ and $|D^{-}_{x} \cup D^{-}_{x+j}|=|D^{-}_{x' } \cup D^{-}_{x'+j}|  $ for any $x, x', j \in \ZZ_r$, we have $ |D^{-}_{r' } \cup D^{-}_{r'+k+1}|=|D^{-}_{-1 } \cup D^{-}_{k}| \ne |D^{-}_{r-1} \cup D^{-}_{k} |$. 
 		% If $ d < n < 2d $, then the claim $|D^{-}_{-1 } \cup D^{-}_{k}| \ne |D^{-}_{r-1} \cup D^{-}_{k} |$ for some $k$,  $0 \le k \le r'- 1$, is not true. example: $ n= 5$, $r=3$, and $ d =4$. $ D({r-1}) = \{ 3,4,5,1 \} $, $ D({-1}) = \{ 2,3,5,1 \} $, $ D(0)= \{ 1234 \} $ and $ D(1) = \{ 2345 \} $; maybe for larger $r$ is is true!
 	\end{proof}

 Let $ k $ be the smallest integer satisfying the conditions in Claim 1 and set $  t = k+2  $. Then $2\le t \le r' +1$ by Claim 1.  

\vspace{.05in}
\textsf{Claim 2:}~ We have $ d(k+1) <n $ and $	|D^{-}_{r-1} \cup (\cup_{i=1}^{k+1} D^{-}_{i-1})| < | \cup_{i= 0 }^{k+1} D^{-}_{r'+i }|$.   
\vspace{-.1in}
\begin{proof} 
	Suppose $d (k+1) \ge n$. Let $j$ be the largest integer such that $d(j+1) < n$. Since $1 \le j < k$, by Claim 1 and the minimality of $k$, $ | D^{-}_{r-1} \cup  D^{-}_{i-1} | = | D^{-}_{r'} \cup  D^{-}_{r'+i }| = | D^{-}_{-1} \cup  D^{-}_{i-1} |$  for every  $1 \le i \le j+1 $. In particular, $ D^{-}_{r-1} \cup D^{-}_{j} = D^{-}_{-1} \cup D^{-}_{j} $.
 	Clearly, $  D^{-}_{-1} \cap  D^{-}_{i-1} =\emptyset $ for $1 \le i \le j$. Denote $s = |D^{-}_{-1} \cap  D^{-}_{j}|$. Then $0 <s <d  $, $  D^{-}_{r-1} \cap  D^{-}_{i-1} =\emptyset $ for $1 \le i \le j $, and $  |D^{-}_{r-1} \cap  D^{-}_{j}| =s $. % (note that $ s<d $ since $  D^{-}_{r-1} \ne D^{-}_{-1} $). and $ s>0 $ otherwise $ d(j+2)<n $ which is not correct by the choice of $ j $
 	 Note that $ \cup_{i=1}^{j }  D^{-}_{i-1} =  \{ l+1,\dots,l+ j d \mod n\} $ and $  D^{-}_{j} =\{ l-d+1,\dots,l \mod n\} $ for some integer $ l $. Since $ D^{-}_{r-1} \cap D^{-}_{j}  = D^{-}_{-1} \cap D^{-}_{j} = \{l-d+1,\dots,l-d+s\mod n\}$, we have $ D^{-}_{r-1} = D^{-}_{-1} =\{l-2d+s+1,\dots,l-d+s\mod n\}$, which contradicts \eqref{eq:f-1}. Therefore, $ d(k+1) <n $. 
 	
	Consider the case $n < d(k+2)$ first. In this case, $(\cup_{i=1}^{k }  D^{-}_{i-1}) \cap (D^{-}_{-1} \cup D^{-}_{k}) = \emptyset$, and $(\cup_{i=1}^{k }  D^{-}_{i-1}) \cap D^{-}_{r-1} = \emptyset$ by Claim 1 and the choice of $ k $. So $ |D^{-}_{-1} \cap  D^{-}_{k}|= d(k+2)-n \le |D^{-}_{r-1} \cap  D^{-}_{k}|$. Since $ | D^{-}_{r-1} \cup  D^{-}_{k} | \ne  | D^{-}_{-1} \cup  D^{-}_{k} |$, it follows that $ |D^{-}_{-1} \cap  D^{-}_{k}| < |D^{-}_{r-1} \cap  D^{-}_{k} |$, or equivalently $ |D^{-}_{-1} \cup  D^{-}_{k}| >|D^{-}_{r-1} \cup  D^{-}_{k} | $, which implies that Claim 2 is true when $n < d(k+2)$.  	
	Now assume $ d(k+2) \le n $. Then $ D^{-}_{r'+i}  \cap D^{-}_{r'+j} = \emptyset$ for $ 0\le i<j \le k +1$, and so  $  | \cup_{i= 0 }^{k+1} D^{-}_{r'+i }|  = d(k+2) $ and $|D^{-}_{r'} \cup D^{-}_{r' + k + 1}|= 2d$. If $|D^{-}_{r-1} \cup (\cup_{i=1}^{k+1} D^{-}_{i-1})| =  d(k+2)$, then $D^{-}_0 , \ldots,D^{-}_k, D^{-}_{r-1}$ are pairwise disjoint and so $|D^{-}_{r'} \cup D^{-}_{r' + k + 1}| = |D^{-}_{r-1} \cup D^{-}_{k}|$, which contradicts Claim 1. Therefore, $|D^{-}_{r-1} \cup (\cup_{i=1}^{k+1} D^{-}_{i-1})| < d(k+2) = | \cup_{i= 0 }^{k+1} D^{-}_{r'+i }| $.   
\end{proof}

%\vspace{.05in}
\textsf{Claim 3:}~$|N_{t-2}| \le  |N'_{t-2}|$.
\vspace{-.1in}
\begin{proof}	  	
	We prove this by showing the existence of an injective mapping from $ N_{t-2} $ to $ N'_{t-2} $. In fact, for any $(\a_t ,x_t) \in N_{t-2}$, we choose a shortest path $P(\a_t,x_t): (\bo,r-1), (\a_1,x_1), \dots,  (\a_t,x_t)$ such that $(\a_{j+1},x_{j+1}) =(\a_j ,x_j \pm 1 )$ or $(\a_{j+1},x_{j+1}) =(\a_j+\e_{i_j} ,x_j)$ and $i_j \in D^{-}_{x_j}$, $0\le j < t $. (Hence the directions of all cube edges on $ P(\a_t,x_t)$ are in $ \cup_{i=0}^{t-1} D^{-}_{x_i} $.) We have $\a_t\ne \bo $, for otherwise the distance between $(\a_t,x_t)$ and $(\bo,r-1)$ would be $ \min\{ r-1-x_t,x_t+1\} \le t-2 $, a contradiction. 
	We construct a path from $(\bo, r')$ to $(\b_t ,y_t) $ as follows. First, set $ y_i := x_i -r+1+ r'  \mod r $ with $ 0 \le y_i<r $, for $ 1\le i \le t $. Since $\max \min\{ |x_i -x_j|, r-|x_i -x_j|\} \le t-2 $, we have $ \max | y_i - y_j | \le t-2 $, where the maximum is taken over all $(i, j)$ such that $ (\a_i,x_i), (\a_j,x_j)\in P(\a_t,x_t)$. This together with $ d(t-1)<n $ (Claim 2) implies that $ D^{-}_{y_1}, D^{-}_{y_2},\dots,D^{-}_{y_t}  $ are mutually disjoint, and so $|\cup_{i=0}^{t-1} D^{-}_{x_i}| \le |\cup_{i=0}^{t-1} D^{-}_{y_i}| = td$. Hence we can choose an injective mapping $g: \cup_{i=0}^{t-1} D^{-}_{x_i} \rightarrow \cup_{i=0}^{t-1} D^{-}_{y_i} $ such that $g(p) \in D^{-}_{y_i}$ for $ p \in  D^{-}_{x_i}$. Set $\a_0=\b_0= \bo $. For $ 0\le i <t $, set $ \b_{i+1} = \b_{i} $ if $ \a_{i+1} = \a_{i} $, and $ \b_{i+1} = \b_{i} + \e_{g(p)} $ if $ \a_{i+1} = \a_{i} + \e_{p}$. 
	In this way we construct a path $P'(\a_t,x_t): (\bo, r'), (\b_1 ,y_1), \dots, (\b_t ,y_t)$. We now prove that $P'(\a_t,x_t)$ is a shortest path. That is, any path $P: (\bo, r'), (\b'_1 ,y'_1), \ldots, (\b'_{t'} ,y'_{t'}) =(\b_t ,y_t) $ in $\rcr$ has length at least $t$. In fact, since $ \b_t\ne \bo$, $P$ contains at least one cube edge and there exist $i, j$ and $l$ such that the $ i $th coordinate of  $ \b_t $ is $ 1$, and $ i \in D^{-}_{y'_l}  \cap D^{-}_{y_j} $ but $ y'_l \ne y_j $. Since the number of ring edges on $P$ is at least $ |y'_l-r'|+| y_t-y'_l|$, it suffices to prove that $ |y'_l-r'|+| y_t-y'_l| \ge t-1$.

In fact, since $ d(t-1)<n $ and $ D^{-}_{r'+i'} \cap D^{-}_{r'+j'} = \emptyset$ for every pair $(i', j')$ with $ |j'-i'| \le t-2$, we have $ | y'_l - y_j | \ge t-1$. Since $ \b_t\ne \bo $, we have
\begin{equation}
 \label{eq:yjyt}
	|y_j-r'|+|y_t-y_j|\le t-1,\, \text{ for }\, 0\le j \le t.
\end{equation}
If $ r'\le y_t \le y_j $ and $ y'_l\le y_j -(t-1) $, then $ r'+y_t \ge 2y_j-(t-1) $ by  \eqref{eq:yjyt} and so $ |y'_l-r'|+| y_t-y'_l| =  r'-y'_l +y_t-y'_l \ge 2y_j-(t-1)-2y'_l \ge t-1 $. If $ y_j\le r'\le y_t  $ and $ y'_l\ge y_j +(t-1) $, then $ -r'-y_t \ge -2y_j-(t-1) $ by (\ref{eq:yjyt}) and so $ |y'_l-r'|+| y_t-y'_l| = y'_l-r'+y'_l-y_t \ge 2y'_l - 2y_j -(t-1) \ge 2(t-1) -(t-1)=t-1 $. Similarly, $ |y'_l-r'|+| y_t-y'_l|\ge t -1$ in all other cases.  

So far we have proved that $P'(\a_t,x_t)$ is a shortest path in $N'_{t-2}$. Since $g$ is an injection, different  $(\a_t ,x_t) \in N_{t-2}$ gives rise to different $(\b_t,y_t) \in N'_{t-2}$. 
\end{proof}  

%\vspace{.1in}
\textsf{Claim 4:}~$|N_{t-1}| < |N'_{t-1}|$.
\vspace{-.1in}
\begin{proof}	
We claim that, for $k \in D^{-}_{r-1} \cup (\cup_{i=1}^{t-1} D^{-}_{i-1})$ and $l  \in \cup_{j=1}^{t} D^{-}_{r-j} $, $(\e_{k}, t-2)$ and $(\e_{l}, r-t+1)$ are in $ N_{t-1} $. In fact, for every $0 \le x \le t-1$,  the paths $(\bo,r-1), (\bo, 0), \dots, (\bo, r-1+x ), (\e_k,  r-1+x ), (\e_k,  r+x ), \dots, (\e_k, t-2)$ and $(\bo,r-1), (\bo, r-2), \dots, (\bo,  r-x-1), (\e_{l },  r-x-1), (\e_{l }, r-x-2), \dots, (\e_{l }, r-t)$ are shortest paths. 
  	Similarly, for $k' \in \cup_{j=0}^{t-1} D^{-}_{r'+j}$ and $l' \in \cup_{j=0}^{t-1} D^{-}_{r'-j}$, $(\e_{k'},r' + t -1 )$ and $(\e_{l'},r'-t+1 )$  are in $ N'_{t-1} $, because for every $ 0\le y \le t-1 $ the paths $(\bo, r'), (\bo, r'+1), \dots, (\bo, r'+y), (\e_{k'}, r'+y), (\e_{k'}, r'+y+1), \dots, (\e_{k'}, r'+t-1)$ and $(\bo, r'), (\bo, r'-1), \dots, (\bo,r'- y), (\e_{l'}, r'- y), (\e_{l'}, r'- y-1), \dots, (\e_{l'}, r'-t+1)$ are shortest paths. 
  	Since  $|\cup_{j=1}^{t} D^{-}_{r-j} | = |\cup_{j=0}^{t-1} D^{-}_{r'-j}|$, the number of $t$-neighbours $(\e_l, r-t)$ of $(\bo,r-1)$ is equal to the number of $t$-neighbours $(\e_{l'}, r'-t+1)$  of $(\bo, r')$. However, by Claim 2, the number $|D^{-}_{r-1} \cup (\cup_{i=1}^{t-1} D^{-}_{i-1})|$ of $t$-neighbours $(\e_{k}, t-2)$ of $(\bo,r-1)$ is strictly less than the number $|\cup_{i=0}^{t-1} D^{-}_{r'+i}|$ of $t$-neighbours $(\e_{k'},r' + t -1 )$ of $(\bo, r')$. 
\end{proof}   
 
%\vspace{.1in}
\textsf{Claim 5:}~$|N_{t}| = |N'_{t}|$.
\vspace{-.1in}
\begin{proof}	
	Clearly, $ N_t $ and $ N'_t $ consist of the vertices on the $ \bo $-ring which are at distance $ t $ from $ (\bo,r-1) $ and $ (\bo,r') $, respectively. The claim follows immediately. %the number of $ t $-neighbours of $(\bo,r-1)$ of the form $(\bo,x_t)$ is equal to the number of $ t $-neighbours of $(\bo, r')$ of the form $(\bo,y_t )$. 
\end{proof}  

Combining Claims 3-5, the number of $t$-neighbours of $(\bo,r-1)$ is strictly less than that of $(\bo, r') $. Therefore, $Q_n^-(d,r)$ is not vertex-transitive and so is not a Cayley graph.  
 	 
\section{Proof of Lemma \ref{lem:Vx}}
\label{app:lemvx}

%\begin{proof} 
	Given non-negative integers $q$ and $p\ge 1$ with $p+q\le z$, let $L_{q,p} $ be the subset of $\ZZ_2^{d(z-1)} $ such that for any $\c \in L_{q,p}$ with $ \hat{w}_{\c} = (w_0,w_1,\ldots,w_s,w_{s+1}) $, we have $  w_{j}=q$, $w_{j+1}=p+q $ for some $ 1\le j \le s-2 $, $w_0 =0 $, $w_{s+1}=z $ and $  w_i\le w_{i+1} $ for $ 0\le i \le s $, where $ s=\|\c\| $. Hence, if $\c \in L_{q,p}$, then $ c_i =0 $ for $ i\in \cup_{j=q+1}^{p+q-1} D(j)$ (that is, $ d(q+1)<i\le d(q+p) $), $ c_i =1$ for at least one $i \in D(q) $ and at least one $ i\in D({p+q}) $. 
	%we have $ q, p+q\in \{ w_0,w_1,\ldots,w_s,w_{s+1} \} $ and $ \{ q+1,\dots, q+ p -1  \} \cap \{ w_1,\ldots,w_s \} = \emptyset$. In other words, for any  $\c \in L_{q,p}$ we have $ \hat{w}_{\c} =(w_0,w_1,\ldots,w_{i-1},q,p+q,w_{i+2},\ldots,w_{s+1} $, for some $ 1\le i \le s-2 $ and $ s=\|\c\| $. 
	Hence $|L_{q,p}| =  2^{d(z-p-1)} (2^{d}-1)$ if $q=0$ or $q = z - p$ and  $|L_{q,p}| =  2^{d(q-1)} (2^{d}-1)^2 2^{d (z-q-p-1) } = (2^{d}-1)^2 2^{d (z-p-2) }$ if $1\le q \le z -p-1$. 
%	For non-negative integers $p$ and $q$, $p+q\le x$, let $A_{q,p}$ be the set of integer partitions of $x$ in which there exists a part equal to $p$ so that the sum of all parts prior to $p$ is equal to $q$. Furthermore, let $L_{q,p} := \{ \c \in \ZZ_2^{d(x-1)} : B'_{\c} \in f(A_{q,p})\}$.
	Thus, for any $p$ and $q$,
\begin{equation}
	\label{eq:lpq}
	| L_{q,p} | \le  2^{d (z- p)}.
	\end{equation}
	For any $ 1\le g <z $, let  $V_{g} := \{ \c \in \ZZ_2^{d(z-1)} : \max_{0\le i \le s} (w_{i+1}  - w_i) \le g \}$ and denote $ p^* = \max_{0\le i \le s} (w_{i+1}  - w_i)$. (Note that $ p^* $ relies on $ (\c,z) $.) For any $\c \in \ZZ_2^{d(z-1)} $ with $p^* >g$, we have $\c\in L_{q,p^*}$ for some $q\ge 0$.  Therefore, by (\ref{eq:lpq}), 
	\begin{align*}
	 |V_{g}| &=  2^{d(z-1)} - | \mathop{\cup}_{ p > g} \mathop{\cup}_{0\le q  \le z - p} L_{q,p}| \\
	 &\ge 2^{d(z-1)} - \sum_{p > g} (z-p+1) |L_{0,p}| \\
	  & \ge  2^{d(z-1)} - \sum_{p=g+1}^{z} (z-p+1) 2^{d(z-p)} \\
	  & = 2^{d(z-1)} - 2^{-d} \sum_{i=1}^{z-g} i 2^{di} . %\ge 2^{d(z-1)} -  (z-g)  2^{d(z-g-1)+1},
	\end{align*}
 	Since $\sum_{i=1}^{k} it^i = t(1-(k+1)t^k + kt^{k+1})/(t-1)^2$, we have $\sum_{i=1}^{k} it^i \le 2 k t^k$ for $t=2 $ and $ \sum_{i=1}^{k} it^i \le kt^{k+2}/(t-1)^2\le  2 k t^k$ for  $t=2^d$ and $d\ge 2$. Hence $ 2^{-d} \sum_{i=1}^{z-g} i 2^{di} \le  (z-g)  2^{d(z-g-1)+1} $ and therefore  
 		$|V_g| \ge 2^{d(z-1)} - (z-g)2^{d(z-g-1)+1}.$
 	Hence, for $g =\lceil \log^2 z \rceil$, we have  $ V_g = V(z) $, $ g \ge \log^2 z $ and so $| V(z) | \ge 2^{d(z-1)} (1 - (z-g) 2^{1-dg})   \ge 2^{d(z-1)} (1 -  2 z^{1-d\log z}) $. 
%\end{proof}  

\section{Proof of Lemma \ref{lem:avr}}
\label{app:lemavr}
 
	From the discussion in Section \ref{sec:tdc2}, $b_k$ is  the number of $k$-compositions $ (r_1,r_2,\dots,r_k)  $ of $m$ with $1\le r_i \le g$ for each $ i $, plus the number of  $(k+1)$-compositions $ (r_1,r_2,\dots,r_{k+1}) $ with solutions $1\le r_i \le g$ for each $ i $. 

	Given non-negative integers $i,q,p$, denote by $C_{q,p}(i)$ the family of $k$-compositions $ (r_1,r_2,\dots,r_k)$ of $m$ such that $\sum_{j<i} r_j = q$,  $r_i= p$ and $\sum_{j>i} r_j = m-q-p$.  
	
	If $i=1$ ($i=k$, respectively), then $q =0$ ($q=m-p$, respectively) and  
	\begin{equation}
	\label{eq:Mipq1}
		|C_{q,p}(i)| = \binom{m-p-1}{k-2}. 
	\end{equation}
	For $2 \le i \le k-1$, we have $i -1 \le q \le m-p-k+i$ and
	\begin{equation}
	\label{eq:Mipq2}
		|C_{q,p}(i)| = \binom{q-1}{i-2} \binom{m -p - q - 1}{k - i -1}. 
	\end{equation}
	%We want to obtain a lower bound for $b_k$.
	Therefore, the set of $k$-compositions $ (r_1,r_2,\dots,r_k)$ with  $\max_{1\le i \le k}r_i \ge g+1$ is  
		$$  \left( \mathop{\cup}_{ p > g} C_{0,p}(1)   \right)  \cup \left( \mathop{\cup}_{ p > g} \mathop{\cup}_{i = 2}^{k-1} \mathop{\cup}_{ q=i -1 }^{ m-p-k+i} C_{q,p}(i)\right) \cup \left( \mathop{\cup}_{ p > g} C_{m-p,p}(k)  \right). $$
	By (\ref{eq:Mipq1}) and (\ref{eq:Mipq2}), the size of this set  is at most
	\begin{equation}
	\label{eq:exprLem1}
		 	2 \sum_{p>g}  \binom{m-p-1}{k-2} +  \sum_{p>g} \sum_{i=2}^{k-1}  \sum_{q= i-1}^{m-p-k+i} \binom{q-1}{i-2} \binom{m - p -q-1}{k - i-1}.
	\end{equation}
	
	To simplify the expression above, we use the following known identity:
	\begin{equation}
	\label{eq:b1}
		\sum_{l=0}^{\alpha } \binom{l}{j} \binom{\alpha  - l}{\beta - j} = \binom{\alpha  + 1}{\beta+ 1},
	\end{equation}	
	where $\alpha,\beta$ and $j$ are integers with $0 \le j \le \beta \le \alpha $. 	 
	Applying (\ref{eq:b1}) to the case where $ \alpha=m-p-2 $, $ \beta = k-3 $ and $ l=q-1 $, one can verify that \eqref{eq:exprLem1} is less than or equal to		 
	$$
	 2 \sum_{p>g} \binom{m-p-1}{k-2} +  \sum_{p>g} \sum_{i=2}^{k-1} \binom{m -p - 1 }{k - 2 } 
	= k \sum_{p=g+1}^{m-k+1}  \binom{m -p - 1 }{k - 2 }= k \sum_{l=k-2}^{m-g-2}  \binom{l }{k - 2 }.
	$$
	Again by (\ref{eq:b1})  applied to $ \alpha = m-g-2 $ and $ j =\beta=k-2$, the right-hand side of the above equation is equal to $  k  \binom{m - g - 1 }{k - 1 } $. So the number of $k$-compositions of $m$ with $1\le r_i \le g$ for $ 1\le i \le k $ is at least $\binom{m-1}{k-1} - k  \binom{m - g - 1 }{k - 1 }.$ Similarly, the number of $(k+1)$-compositions of $m$ with $1\le r_i \le g$ for $ 1\le i \le k+1 $ is at least $	\binom{m-1}{k} - (k+1)  \binom{m - g - 1 }{k  }.$   
	Thus, since $\binom{m-1}{k-1} + \binom{m-1}{k} = \binom{m}{k}$, we have $ b_k\ge \binom{m}{k} - (k+1)  \binom{m - g  }{k  } $. On the other hand, the number of $k$-compositions of $m$ plus the number of  $(k+1)$-compositions of $ m $ is an upper bound for $ b_k$, that is, $ b_k \le \binom{m-1}{k-1} + \binom{m-1}{k} =\binom{m}{k}$. Therefore, we have
	\begin{align}
	\label{eq:ubk}
		\sum_{k=0}^{m}\! \binom{m}{k} {z}^k -\! \sum_{k=0}^{\lceil m/g \rceil-1}\! \binom{m}{k} {z}^k -\! \sum_{k=\lceil m/g \rceil}^{m} \!(k+1) \binom{\!m\!-\!g\!}{k } {z}^k &\le
		\sum_{k=\lceil m/g \rceil}^{m} b_k {z}^k  \\  \label{eq:ubk2}
		&\le \sum_{k=0 }^m \binom{m}{k} {z}^k  	
		= ( {z} +1 )^m.
	\end{align}
	Note that 	
	\begin{align}		 	 
	\nonumber
		\sum_{k = \lceil m/g \rceil}^{m} (k+1)  \binom{m - g  }{k} {z}^k  % = \sum_{k =  \lceil m/g \rceil}^{m-g-1} (k+1) {z}^k    \binom{m - g-1 }{k}  	
		&\le (m-g+1 ) \sum_{k =  \lceil m/g \rceil}^{m-g}   \binom{m - g  }{k}   {z}^k\\
			\label{eq:psum2}	
		&<  
		({m-g+1}) ({z}+1)^{m-g}.
	\end{align} 
	Using $ \sum_{k=0}^{t}  \binom{m}{k} \le \binom{m}{t} \frac{m-t+1}{m-2 t +1} $ and $ \binom{m}{t} \le (\frac{me}{t})^t $, we obtain  
	\begin{align} 
	\sum_{k = 0}^{ \lceil m/g \rceil -1 } {z}^k  \binom{m}{k} 
	&\le  {z}^{ \lceil \frac{m}{g}  \rceil  - 1 } \sum_{k = 0}^{ \lceil m/g \rceil - 1 }  \binom{m}{k} \\
	&\le  {z}^{ \frac{m}{g}  }   \frac{( m  e  )^{\frac{m}{g} } }{ (\frac{m}{g} )^{\frac{m}{g} } } \frac{m-\lceil \frac{m}{g}  \rceil }{m-2\lceil \frac{m}{g}  \rceil -1 }.
	\end{align}
	This together with (\ref{eq:psum2}) and the lower bound in (\ref{eq:ubk}) yields  
	\[
		\sum_{k =\lceil m/g \rceil }^m b_k {z}^k \ge %\sum_{k =  \lceil m/g \rceil}^{m} {z}^k   \binom{m}{k} - \sum_{k =  \lceil m/g \rceil}^{m} {z}^k  (k+1)  \binom{m - g - 1 }{k  }   =
		({z}+1)^m \left( 1 -  \left(\frac{z}{z+1}\right)^{\frac{m}{g} }\frac{(e g)^{\frac{m}{g} }}{  (({z}+1)^{g-1})^{\frac{m}{g} } }\frac{m-\lceil\frac{m}{g}\rceil }{m-2\lceil\frac{m}{g}\rceil -1 }  - \frac{m-g}{({z} +1)^g} \right). 
	\]
	One can verify that $ \left(\frac{z}{z+1}\right)^{\frac{1}{g} }\frac{(e g)^{\frac{1}{g} }}{  (({z}+1)^{1-1/g})} \frac{(m-\lceil\frac{m}{g}\rceil)^{1/m}}{(m-2\lceil\frac{m}{g}\rceil -1)^{1/m}}   \le  \frac{g^{1/m}}{(z+1)^{g/m}}$ for $ m\ge9 $. % 
	Using this, the expression above follows that
	\[
		\sum_{k =\lceil m/g \rceil }^m b_k {z}^k \ge
			({z}+1)^m \left( 1 - ({2}/(z +1))^g \right). 
	\]
	This together with the upper bound in (\ref{eq:ubk2}) completes the proof.

\bigskip
{\bf Acknowledgements}~~Mokhtar was supported by MIFRS and MIRS of the University of Melbourne and acknowledges helpful discussions with Teng Fang.  Zhou was supported by the Australian Research Council (FT110100629).

%%%%%%%%%%%%%%% bibliography %%% bibliography %%% bibliography %%%%
 
\footnotesize 
 
 \providecommand{\MR}{\relax\ifhmode\unskip\space\fi MR }
 \providecommand{\MRhref}[2]{%
   \href{http://www.ams.org/mathscinet-getitem?mr=#1}{#2}
 }
 \providecommand{\href}[2]{#2}


\begin{thebibliography}{9}
 	
 	
 	\bibitem{alperin1995}
 	J.L. Alperin and R.B. Bell, \emph{ Groups and
 		Representations, Graduate Texts in Mathematics}, vol.~162, Springer, New York,
 	1995. 
 	
 	\bibitem{bermond2000}
 	J.C. Bermond, L.~Gargano, S.~Perennes, A.A. Rescigno, and U.~Vaccaro,
 	\emph{Efficient collective communication in optical networks}, Theoret. Comput. Sci. \textbf{223} (2000),   165--189.
 	
 	\bibitem{bonchev2002}
 	D.~Bonchev, {The {W}iener number--some applications and new
 		developments}, D.H. Rouvray and R.B. King (Eds.), \emph{Topology in Chemistry},
 	Woodhead Publishing, Chichester, 2002, pp.~58--88.
 	
 	\bibitem{choi2008}
 	D.~Choi, O.~Lee, and I.~Chung, \emph{A parallel routing algorithm on recursive
 		cube of rings networks employing {H}amiltonian circuit {L}atin square},
 	Inform. Sci. \textbf{178} (2008), no.~6,   1533--1541.
 	
 	\bibitem{chung1987}
 	F.~Chung, E.~Coffman, M.~Reiman, and B.~Simon, \emph{The forwarding index of
 		communication networks}, IEEE Trans. Inform. Theory \textbf{33}
 	(1987), no.~2,   224--232.
 	
 	\bibitem{cortina1998}
 	T.J.~Cortina and Z.~Xu, \emph{The cube-of-rings interconnection network},
 	Internat. J. Found. Comput. Sci. \textbf{9} (1998), no.~01,   25--37.
 	
 	\bibitem{dally2004}
 	W.J.~Dally and B.P.~Towles, \emph{Principles and Practices of Interconnection
 		Networks}, Elsevier, 2004.
 	
 	\bibitem{Fang2012}
 	X.G.~Fang and S.~Zhou, \emph{Gossiping and routing in second-kind {F}robenius
 		graphs}, European J. Combin. \textbf{33} (2012), no.~6,   1001--1014.
 	
 	\bibitem{fris1997}
 	I.~Fri{\v{s}}, I.~Havel, and P.~Liebl, \emph{The diameter of the cube-connected
 		cycles}, Inform. Process. Lett. \textbf{61} (1997), no.~3,   157--160.
 	
 	\bibitem{gan}
 	H.-S.~Gan, H.~Mokhtar, and S.~Zhou, \emph{Forwarding and optical indices of
 		4-regular circulant networks},  J. Discrete Algorithms \textbf{35} (2015),
 	27--39.
 	
 	\bibitem{gauyacq1997}
 	G.~Gauyacq, \emph{Edge-forwarding index of star graphs and other {C}ayley
 		graphs}, Discrete Appl. Math. \textbf{80} (1997), no.~2,   149--160.
 	
 	\bibitem{heydemann1997}
 	M.C.~Heydemann, {Cayley graphs and interconnection networks}, G.~Hahn and G.~Sabidussi (Eds.),  \emph{Graph Symmetry}, Nato Asi Series, vol. 497, Springer, Netherlands, 1997.
 	
 	\bibitem{heydemann1989}
 	M.C.~Heydemann, J.C.~Meyer, and D.~Sotteau, \emph{On forwarding indices of
 		networks}, Discrete Appl. Math. \textbf{23} (1989), no.~2,   103--123.
 	
 	\bibitem{heydemann1994}
 	M.C.~Heydemann, J.C.~Meyer, D.~Sotteau, and J.~Opatrny, \emph{Forwarding
 		indices of consistent routings and their complexity}, Networks \textbf{24}
 	(1994), no.~2,   75--82.
 	
 	\bibitem{hu2005}
 	H.~Hu, N.~Gu, and J.~Cao, \emph{A note on recursive cube of rings network},
 	IEEE Trans. Parallel Distrib. Systems \textbf{16} (2005), no.~10, 1007--1008.
 	
 	\bibitem{kosowski2009}
 	A.~Kosowski, \emph{Forwarding and optical indices of a graph}, Discrete Appl.
 	Math. \textbf{157} (2009), no.~2,   321--329.
 	
 	\bibitem{lai2010}
 	P.L.~Lai, H.C.~Hsu, C.H.~Tsai, and I.A.~Stewart, \emph{A class of hierarchical
 		graphs as topologies for interconnection networks}, Theoret. Comput. Sci.  \textbf{411} (2010), no.~31,   2912--2924.
 	
 	\bibitem{laksh1993}
 	S.~Lakshmivarahan, J.S.~Jwo, and S.K.~Dhall, \emph{Symmetry in interconnection
 		networks based on {C}ayley graphs of permutation groups: {A} survey},  Parallel Comput. \textbf{19} (1993), no.~4,   361--407.
 	
 	\bibitem{mohar1997}
 	B.~Mohar, Some applications of Laplace eigenvalues of graphs, G.~Hahn and G.~Sabidussi (Eds.),  \emph{Graph Symmetry}, Nato Asi Series, vol. 497, Springer, Netherlands, 1997. 
 	
 	\bibitem{mans2004}
 	B.~Mans and I.~Shparlinski, {Bisecting and gossiping in circulant graphs}, M.~Farach-Colton (Eds.), Lecture Notes in Comput. Sci., vol.~{2976}, Springer, Berlin, Heidelberg, 2004, pp.~589--598.  
 	
 	\bibitem{park2000}
 	J.H.~Park and K.Y.~Chwa, \emph{Recursive circulants and their embeddings among
 		hypercubes}, Theoret. Comput. Sci. \textbf{244} (2000), no.~1,   35--62.
 	
 	\bibitem{preparata1981}
 	F.P.~Preparata and J.~Vuillemin, \emph{The cube-connected cycles: {A} versatile
 		network for parallel computation}, Comm. ACM \textbf{24}
 	(1981), no.~5,   300--309.
 	
 	\bibitem{saad1993}
 	R.~Saad, \emph{Complexity of the forwarding index problem}, SIAM J. Discrete
 	Math. \textbf{6} (1993), no.~3,   418--427.
 	
 	\bibitem{shahrokhi1995}
 	F.~Shahrokhi and L.A.~Sz{\'e}kely, \emph{On group invariant flows and
 		applications}, Y.~Alavi, A.~Schwenk (Eds.), Graph Theory, Combinatorics, and Applications: Proc. of 7th
 	Quad. Int. Conf. on Theory and Applications of Graphs, vol.~2,   1033--1042, John Wiley and Sons, Chichester (1995). 
 	
 	\bibitem{shah2001}
 	F.~Shahrokhi and L.A. Sz{\'{e}}kely, \emph{Constructing integral uniform flows
 		in symmetric networks with application to the edge-forwarding index problem},
 	Discrete Appl. Math. \textbf{108} (2001), no.~1–2,   175--191.
 	
 	\bibitem{sole1994}
 	P.~Sol{\'e}, \emph{The edge-forwarding index of orbital regular graphs},
 	Discrete Math. \textbf{130} (1994), no.~1,   171--176.
 	
 	\bibitem{sun2000}
 	Y.~Sun, P.Y.S.~Cheung, and X.~Lin, \emph{Recursive cube of rings: a new
 		topology for interconnection networks}, 
 	IEEE Trans. Parallel Distrib. Systems \textbf{11} (2000), no.~3,   275--286.
 	
 	\bibitem{thomson2008}
 	A.~Thomson and S.~Zhou, \emph{Frobenius circulant graphs of valency four}, J. Aust. Math. Soc. \textbf{85} (2008),   269--282.
 	
 	\bibitem{thomson2010}
 	A.~Thomson and S.~Zhou, \emph{Gossiping and routing in undirected triple-loop networks},
 	Networks \textbf{55} (2010), no.~4,   341--349.
 	
 	\bibitem{thomson2014}
 	A.~Thomson and S.~Zhou, \emph{Frobenius circulant graphs of valency six, {E}isenstein--{J}acobi
 		networks, and hexagonal meshes}, European J. Combin. \textbf{38} (2014),   61--78.
 	
 	\bibitem{xie2013}
 	K.~Xie, J.~Li, Y.~Wang, and C.~Yuen, \emph{Recursive-cube-of-rings {(RCR})
 		revisited: Properties and enhancement}, preprint (2013), available at http://arxiv.org/abs/1305.2214.
 	
 	\bibitem{yan2009}
 	J.~Yan, J.M.~Xu, and C.~Yang, \emph{Forwarding index of cube-connected cycles},
 	Discrete Appl. Math. \textbf{157} (2009), no.~1, 1--7.
 	
 	\bibitem{zhou2009}
 	S.~Zhou, \emph{A class of arc-transitive {C}ayley graphs as models for
 		interconnection networks}, SIAM J. Discrete Math. \textbf{23} (2009), no.~2,   694--714.
 	
 	\bibitem{zhou2006}
 	S.~Zhou, N.~Du, and B.~Chen, \emph{A new family of interconnection networks of
 		odd fixed degrees}, J. Parallel Distrib. Comput. \textbf{66} (2006), no.~5,   698--704.
 	
 \end{thebibliography}
\end{document}